\documentclass[12pt,twoside]{article}
\usepackage{amssymb}
\usepackage{amsmath}
\usepackage{cases}
\usepackage{amsfonts}
\usepackage{amsmath,amssymb}
\usepackage{multicol}
\usepackage{color}
\usepackage{graphicx}
\usepackage{graphics}
\usepackage{subfig}
\usepackage{float}
\usepackage[colorlinks=true,
linkcolor=cyan,
filecolor=blue,
urlcolor=red,
citecolor=green]{hyperref}

\def\R{\mathbb{R}}

\def\epsilon{\varepsilon}
\renewcommand\tilde{\widetilde}
\newcommand{\be}{\begin{equation}}
\newcommand{\ee}{\end{equation}}
\newcommand{\baa}{\begin{array}}
\newcommand{\eaa}{\end{array}}
\newcommand{\ba}{\begin{eqnarray}}
\newcommand{\ea}{\end{eqnarray}}

\newtheorem{theorem}{Theorem}[section]
\newtheorem{lemma}[theorem]{Lemma}

\newtheorem{definition}[theorem]{Definition}

\numberwithin{equation}{section}
\newenvironment{proof}[1][Proof]{\noindent\textbf{#1.} }{\hfill $\Box$}
\allowdisplaybreaks

\makeatletter
\begin{document}
\date{}
\title{\bf{Curved fronts of  bistable reaction-diffusion equations in spatially periodic media:  $N\ge 2$}}
\author{ Hongjun Guo$^{\dagger}$,  Haijian Wang$^{\dagger}$\\
		\\
		\footnotesize{ $^{\dagger}$ School of Mathematical Sciences, Key Laboratory of Intelligent Computing and }\\
		\footnotesize{Applications (Ministry of Education), Institute for Advanced Study, Tongji University, Shanghai, China}
		}
\maketitle

\begin{abstract}\noindent{}

This paper is concerned with curved fronts of bistable reaction-diffusion equations in spatially periodic media for dimensions $N\geq 2$. The curved fronts concerned are transition fronts connecting $0$ and $1$. Under a priori assumption that there exist moving pulsating fronts in every direction, we show the existence of polytope-like curved fronts with $0$-zone being a polytope and $1$-zone being the complementary set. By reversing some conditions, we also show the existence of curved fronts with reversed $0$-zone and $1$-zone. Furthermore, the curved fronts constructed by us are proved to be unique and asymptotic stable. 

\end{abstract}


\section{Introduction}
In this paper, we consider the following spatially periodic reaction-diffusion equation
\begin{equation}\label{1.1}
u_t-\nabla \cdot(A(x)\nabla  u)=f(x,u),\quad (t,x)\in \mathbb{R}\times \mathbb{R}^N, 
\end{equation}
where $N\geq 2$. The reaction term $f(x,u)$ and the coefficient matrix $A(x)$ are assumed to satisfy the following hypotheses:
\begin{itemize}
\item[(A1)] $f(x,u)$ is of $C^{\alpha}$ ($0<\alpha<1$) in $x\in\mathbb{R}^N$ uniformly for $u\in[0,1]$, and of $C^2$ in $u\in[0,1]$ uniformly for $x\in\mathbb{R}^N$.

\item[(A2)] $f(x,u)$ is $L$-periodic with respect to $x$, that is, there exists $L=(L_1,\cdots,L_N)\in\mathbb{R}^N$ such that,
$$f(x_1+k_1L_1,\cdots,x_n+k_NL_N,u)=f(x_1,\cdots,x_N,u)$$
for any $(x_1,\cdots,x_N,u)\in\mathbb{R}^N\times\mathbb{R}$ and $(k_1,\cdots,k_N)\in\mathbb{Z}^N$.

\item[(A3)] $f(x,u)$ is of bistable type in $u$, that is, for every $x\in\mathbb{R}^N$, $f(x,0)=f(x,1)=0$, $f(x,u)<0$ in $u\in(0,\theta_x)$ and $f(x,u)>0$ in $u\in(\theta_x,1)$ for some periodic function $\theta_x\in(0,1)$. In addition, there exist $\kappa>0$ and $\sigma\in(0,1/2)$ such that $-f_u(x,u)\geq\kappa$ for $(x,u)\in\mathbb{R}^N\times [0,\sigma]\cup \mathbb{R}^N\times [1-\sigma,1]$.

\item[(A4)] $A(x)=(a_{ij}(x))_{1\leq i,j\leq N}$ is $L$-periodic, smooth, positive definite and there exist $0<\lambda_1\leq\lambda_2$ such that
\begin{equation*}
\lambda_1|\xi|^2\leq \sum_{i,j=1}^N  a_{ij}(x)\xi_i\xi_j\leq\lambda_2|\xi|^2 
\end{equation*}
for all $x\in\mathbb{R}^N$ and $\xi\in\mathbb{R}^N$.
\end{itemize}

Such an equation and its traveling wave solutions have been studied extensively due to the development of mathematical models in physics, chemistry and biology. See \cite{1937THE} for the reaction-diffusion model describing the propagation of advantageous genes and see \cite{MR3029132} for the traveling waves in the Belousov-Zhabotinsky reactions.

For mathematical convenience, we extend $f(x,u)$ to $\mathbb{R}^N\times\mathbb{R}$ such that
\begin{equation*}
-f_u(x,u)\geq\kappa\text{ for all }(x,u)\in\mathbb{R}^N\times((-\infty,\sigma]\cup[1-\sigma,+\infty)).
\end{equation*}
Then, $f(x,u)$ is globally Lipschitz continuous in $u\in\mathbb{R}$ uniformly for $x\in\mathbb{R}^N$. A typical example of $f(x,u)$ is the cubic function, that is,
\begin{equation*}
f(x,u)=u(1-u)(u-\theta(x)),
\end{equation*}
where $\theta(x):\mathbb{R}^N\to (0,1)$ is a H\"{o}lder continuous and $L$-periodic function.

Before presenting our results, we first recall some known results about planar and curved fronts of reaction-diffusion equations. For the following homogeneous bistable equation
\begin{equation}\label{homogeneous case}
u_t-\Delta u=f(u),\quad (t,x)\in \mathbb{R}\times \mathbb{R}^N,
\end{equation}
with $f$ satisfying
$$f(0)=f(1)=0,\, f(u)<0 \hbox{ for $u\in (0,\theta)$ and } f(u)>0 \hbox{ for $u\in (\theta,1)$}, \hbox{ with some $\theta\in(0,1)$},$$
it follows from the well-known paper of Fife and McLeod \cite{1977The} that there exists a unique (up to shifts) traveling front $\phi(x\cdot e-c_ft)$ for every $e\in\mathbb{S}^{N-1}$ satisfying $0<\phi<1$, $\phi'<0$, $ \phi(-\infty)=1$ and $\phi(+\infty)=0$, where $e\in\mathbb{S}^{N-1}$ denotes the propagation direction and $c_f$ denotes the propagation speed, which is uniquely determined by $f$ and has the same sign of $\int_0^1 f(s)ds$. Solutions of this form are called planar fronts, since all level sets of $\phi(x\cdot e-c_ft)$ are hyperplanes. 

In addition to planar fronts, there exist many types of curved fronts in high dimensional spaces. For $N\ge 2$ and $c_f>0$, Hamel, Monneau and Roquejoffre \cite{MR2166719} studied  the existence and qualitative properties of conical-shaped fronts $u(t,x)=\phi(|x'|,x_N-ct)$ of \eqref{homogeneous case}, where $x'=(x_1,\dots,x_{N-1})$ and $c>c_f$. Let $\alpha\in(0,\pi/2)$ be a  given angle, and then the conical-shaped front satisfies
\begin{equation*}
	\left\{
	\begin{aligned}
		&\phi(r,z)\to 1\text{ (resp. 0 ) unif. as }z-\psi(r)\to -\infty\text{ (resp. $+\infty$)},\\
		&c=\frac{c_f}{\sin\alpha}\text{ and }\psi'(+\infty)=\cot\alpha,\\		
	\end{aligned}
	\right.
\end{equation*}
for some $C^1$ function $\psi:[0,+\infty)\to\mathbb{R}.$ For $N=2$, Ninomiya and Taniguchi \cite{MR2139343,MR2220750} also have shown the existence, uniqueness and global stability of the conical-shaped front.  This front in dimension $2$ is also called V-shaped front since the level sets of the front are like V-shaped curves having  asymptotic lines, that is, 
\begin{equation*}
	\psi(x_1)-\cot\alpha|x_1|\to 0, \text{ as }|x_1|\to +\infty.
\end{equation*}
Moreover, it has been proved that the V-shaped fronts are the only type of traveling fronts for \eqref{homogeneous case} in dimension $2$, see \cite{Gui2012} and \cite{MR2166719}. For high dimensions, more curved fronts including non-axisymmetric ones are also known to exist, see \cite{Ninomiya2021,MR2838366,T2015,MR2318388,MR2494701, Taniguchi2012}.  For the case $c_f=0$, the conical-shaped fronts cannot exist, see \cite{HM2000}, but there exist exponential shaped fronts ($N=2$) and parabolic shaped fronts ($N\ge 3$), see \cite{CGHNR}. We also refer to serial works of Taniguchi \cite{T2019,T2020,T2024} for the existence of axisymmetric and non-axisymmetric fronts in the case $c_f=0$.

Besides the homogeneous case,  front-like solutions in heterogeneous media have also attracted a lot of attention in recent years. For instance, Berestycki, Hamel and Matano \cite{BHM} constructed almost planar fronts in exterior domains and Guo and Monobe \cite{MR4211100} proved the existence of almost V-shaped fronts later on. Other researches regarding to fronts in exterior domains plus time-periodic environment, spatially periodic environment, nonlocal dispersion and so on, can be referred to \cite{MR4656878,Li2021,QLS}.  Front-like solutions also exist in cylinders or cylinder-like domains, see \cite{BBC,CG,GHS,MR4517350,MR2517769}. We also refer to \cite{MR3437538,MR3876557,MR2833431,MR4334733,MR4440469,MR2771259,MR4518515,MR4593934} for the existence, uniqueness and stability of curved fronts for equations with periodic reaction terms.

The traveling fronts and front-like solutions actually share some common features, that is, they all converge to the stable states $0$ and $1$ far away from their moving or stationary level sets, uniformly in time. This observation led to the introduction of a more general notion of traveling fronts, that is, transition fronts, which were introduced by Berestycki and Hamel in \cite{MR2898886}. We recall some notations for transition fronts here. First, let $d(A,B)$ denote the Euclidean distance between any two subsets $A$ and $B$ of $\mathbb{R}^N$, that is, $d(A,B)=\inf\{|x-y|;(x,y)\in A\times B\}$, where $|\cdot|$ is the Euclidean norm in $\mathbb{R}^N$. For $x\in\mathbb{R}^N$ and $r>0$, let $B(x,r)=\{y\in\mathbb{R}^N; |x-y|\le r\}$.  Consider two families $(\Omega_{t}^{-})_{t\in\mathbb{R}}$ and $(\Omega_{t}^{+})_{t\in\mathbb{R}}$ of open nonempty subsets of $\mathbb{R}^N$ satisfying
\begin{equation}\label{transition1}
	\forall t\in\mathbb{R},\quad	\left\{
	\begin{aligned}
		&\Omega_{t}^{-}\cap\Omega_{t}^{+}=\emptyset,\\
		&\partial\Omega_{t}^{-}=\partial\Omega_{t}^{+}=:\Gamma_t,\\
		&\Omega_{t}^{-}\cup\Gamma_t\cup\Omega_{t}^{+}=\mathbb{R}^N,\\
		&\sup\{d(x,\Gamma_t);x\in\Omega_{t}^{+}\}=\sup\{d(x,\Gamma_t);x\in\Omega_{t}^{-}\}=+\infty.
	\end{aligned}
	\right.
\end{equation}
and
\begin{equation}\label{transition2}
\left\{\begin{aligned}
	&\inf\{\sup\{d(y,\Gamma_t);y\in\Omega_{t}^{+}\cap B(x,r)\}; t\in\mathbb{R}; x\in\Gamma_t\}\to+\infty,\\
	&\inf\{\sup\{d(y,\Gamma_t);y\in\Omega_{t}^{-}\cap B(x,r)\}; t\in\mathbb{R}; x\in\Gamma_t\}\to+\infty,
	\end{aligned}
	\quad\text{ as }r\to+\infty.
	\right.
\end{equation}
The above conditions imply that $\Gamma_t$ splits $\mathbb{R}^N$ into two unbounded parts $\Omega_{t}^{-}$ and $\Omega_{t}^{+}$, and for each $t\in\mathbb{R}$, sets $\Omega_{t}^{\pm}$ are assumed to contain points which could be infinitely far from $\Gamma_t$. The condition \eqref{transition2} means that any point on $\Gamma_t$ is not too far from the centers of two large balls  included in $\Omega_t^-$ and $\Omega_t^+$, this property being in some sense uniform with respect to $t$ and to the point on $\Gamma_t$. Furthermore, it was assumed in \cite{MR2898886} that the interfaces $\Gamma_t$ are made of a finite number of graphs, that is, there is an integer $n\geq 1$ such that, for each $t\in\mathbb{R}$, there are $n$ open subsets $\omega_{i,t}\subset\mathbb{R}^{N-1}$, $n$ continuous maps $\psi_{i,t}:\omega_{i,t}\to\mathbb{R}$ and $n$ rotations $R_{i,t}$ of $\mathbb{R}^N$ (for all $1\leq i \leq n$), such that
\begin{equation}\label{transition3}
	\Gamma_t\subset\bigcup_{1\leq i \leq n}R_{i,t}(\{x\in\mathbb{R}^N; (x_1,\dots,x_{N-1})\in\omega_{i,t}, x_N=\psi_{i,t}(x_1,\dots,x_{N-1})\}).
\end{equation}

\begin{definition}
	\cite{MR2898886} For problem \eqref{1.1}, a transtion front connecting $0$ and $1$ is a classical (time-global) solution $u$ of  \eqref{1.1} such that $u\not\equiv0,1$, and there exist some sets $(\Omega_{t}^{\pm})_{t\in\mathbb{R}}$ and $(\Gamma_t)_{t\in\mathbb{R}}$ satisfying \eqref{transition1}, \eqref{transition2},\eqref{transition3}, and for every $\varepsilon>0$, there exists $M>0$ such that
	\begin{eqnarray}\label{transition4}
		\left\{
		\begin{array}{lll}
			\forall t\in\mathbb{R},\, \forall x\in\Omega_{t}^{+}, &&(d(x,\Gamma_t)\geq M)\Rightarrow(u(t,x)\geq 1-\varepsilon),\\
			\forall t\in\mathbb{R},\, \forall x\in\Omega_{t}^{-}, &&(d(x,\Gamma_t)\geq M)\Rightarrow(u(t,x)\leq \varepsilon).	
		\end{array}
		\right.
	\end{eqnarray}
	Furthermore, $u$ is said to have a global mean speed $\gamma$ ($\ge 0$) if
	$$\frac{d(\Gamma_t,\Gamma_s)}{|t-s|}\rightarrow \gamma  \hbox{ as } |t-s|\rightarrow +\infty.$$
\end{definition}


Because of \eqref{transition4}, the interfaces $\Gamma_t$ represent the location of the transition front in some sense. Thus, the global mean speed is defined by the limiting normal speed of $\Gamma_t$. One knows from \cite{MR2898886} that the global mean speed is unique if it exists, and it is independent of the choice of families $(\Omega_t^{\pm})_{t\in\mathbb{R}}$ and $(\Gamma_t)_{t\in\mathbb{R}}$. Also by \eqref{transition4}, the shapes of interfaces $\Gamma_t$ reflect the profiles of transition fronts in some sense.

Since interfaces of a transition front can be selected as its level sets, we know from what has been mentioned earlier that in the homogeneous case, there are various possible shapes of interfaces, such as conical shape, V-shape, pyramidal shape, and so on. Recently, Guo and Wang \cite{preprint1}  constructed transition fronts by mixing finitely many planar fronts for which interfaces are polytope-like and changing in time. It is worth mentioning that Ninomiya and Taniguchi \cite{NT2024} also studied polyhedral entire solutions in a different way.

In this paper, we focus on spatially periodic case and aim to show their interfaces could also be polytope-like but not changing in time for our case. The transition fronts constructed by us are more like pyramidal fronts in the homogeneous case. The idea is still mixing several planar fronts as in the homogeneous case, but pulsating fronts instead in spatially periodic case. However, because the profiles and speeds of pulsating fronts are anisotropic unlike planar fronts in homogeneous case, we have to overcome more difficulties and the existence of ploytope-like transition fronts has to be restricted under some conditions.

 We now recall the definition of pulsating front by referring to \cite{MR1900178,MR850456}.

\begin{definition}\label{pulsating-def}
Let $\mathbb{T}^N=[0,L_1]\times[0,L_2]\times\cdots\times[0,L_N]$ be a periodic cell. A pair $(U_e,c_e)$ with $U_e : \mathbb{R}\times\mathbb{T}^N\to\mathbb{R}$ and $c_e\in\mathbb{R}$ is said to be a pulsating front of (\ref{1.1}) with effective speed $c_e$ in the direction $e\in\mathbb{S}^{N-1}$ connecting 0 and 1 if the following conditions are satisfied:

(i) The map $u(t,x):=U_e(x\cdot e-c_et,x)$ is an entire (time-global) classical solution of the equation (\ref{1.1}).

(ii) For every $\xi\in\mathbb{R}$, the profile $U_e(\xi,x)$ is L-periodic with respect to x and satisfies
\begin{equation*}
\lim\limits_{\xi\to+\infty}U_e(\xi,x)=0,\lim\limits_{\xi\to-\infty}U_e(\xi,x)=1,\text{ uniformly for }x\in\mathbb{R}^N.
\end{equation*}
\end{definition}

For the existence of pulsating fronts, we refer to \cite{MR3689331,MR3552256,MR3240934,MR3420507, MR4656878, MR2526414, MR1129560} for some existence results, and refer to \cite{MR3689331,MR1251222,MR3724753} for some nonexistence results.

In dimension $2$,  Guo et al. \cite{MR4334733} have studied curved fronts of \eqref{1.1} when $A(x)\equiv I$.
They constructed curved fronts by mixing two pulsating fronts coming from two different directions under some conditions on variation of speed $c_e$ with respect to directions $e$. Such conditions are shown to be nearly necessary.

In this paper, we extend the results of \cite{MR4334733} to dimensions $N\ge 3$ by mixing finitely many pulsating fronts. As in \cite{MR4334733}, we need to assume a priori that
\begin{itemize}
\item[(H1)] $\int_{\mathbb{T}^N\times[0,1]}f(x,u)dxdu>0$,

\item[(H2)] For every direction $e\in\mathbb{S}^{N-1}$, the equation (\ref{1.1}) admits a moving pulsating front $U_e(x\cdot e-c_et,x)$ with speed $c_e\neq 0$.
\end{itemize}
This immediately implies $c_e>0$, since the speed $c_e$ has the same sign of $\int_{\mathbb{T}^N\times[0,1]}f(x,u)dxdu$ under the above assumptions, see \cite{MR3552256, MR3772873}. We point out that (H1) is not lacking generality. If $\int_{\mathbb{T}^N\times[0,1]}f(x,u)dxdu<0$, one can deduce similar results by considering  $\widetilde{U}_e(\xi,x)=1-U_e(-\xi,x)$ instead of $U_e(\xi,x)$ and considering $\widetilde{U}_e(\xi,x)$ with speed $-c_e$. In addition, it follows from \cite{MR2898886} and \cite{MR3772873} that the propagation speed $c_e$ and the profile $U_e(\xi,x)$ are unique (up tp shifts), and it satisfies  $\partial_{\xi}U_e(\xi,x)<0$ for every $e\in\mathbb{S}^{N-1}$. We normalize $U_e(\xi,x)$ by $\int_{\R^+\times\mathbb{T}^N} U_e^2(\xi,y)dyd\xi=1$ for all $e\in\mathbb{S}^{N-1}$ and we always assume that (A1)-(A4) and (H1)-(H2) hold throughout this paper.

For convenience of narration, we introduce some geometric notations in high dimensional spaces.  Take $e_0\in\mathbb{S}^{N-1}$ and $n\geq 2$ unit vectors $e_i\in\mathbb{S}^{N-1}$, $i=1,2\cdots,n$ such that $e_i\neq e_j$ for any $i\neq j$ and $e_i\cdot e_0>0$ for all $i\in\{1,\cdots,n\}$. Let $Q_i$ be the hyperplane determined by $e_i$, that is,
$$Q_i=\{x\in\mathbb{R}^N;\, x\cdot e_i=0\}.$$
Since $e_i\neq e_j$ for $i\neq j$, one knows that $Q_i$ and $Q_j$ are not parallel. Let $\mathcal{Q}$ be the polytope enclosed by $Q_1$, $Q_2$, $\cdots$, $Q_n$, that is, 
$$\mathcal{Q}=\{x\in\mathbb{R}^N;\, \min\limits_{1\leq i\leq n}x\cdot e_i> 0\}.$$
This polytope is unbounded since $e_i\cdot e_0>0$ for all $i\in\{1,\cdots,n\}$.
Let $\partial\mathcal{Q}=\{x\in\mathbb{R}^N;\, \min\limits_{1\leq i\leq n}x\cdot e_i= 0\}$ be the boundary of $\mathcal{Q}$. The joint part of $\partial\mathcal{Q}$ and $Q_i$ is called a facet of the polytope, denoted by $\widetilde{Q}_i=\partial\mathcal{Q}\cap Q_i$. Let $\mathcal{R}_{ij}=\widetilde{Q}_i\cap\widetilde{Q}_j$ be the ridges for $i\neq j$, which represents the intersection of $\widetilde{Q}_i$ and $\widetilde{Q}_j$. Let $\mathcal{R}$ be the set of all ridges of $\mathcal{Q}$.

Now we are ready to present our results. By (H2), there exists a pulsating front in the direction $e_i\in \mathbb{S}^{N-1}$ for every $i\in \{1,\cdots ,n\}$, that is, $U_{e_i}(x\cdot e_i-c_{e_i}t,x)$.
For $n$ vectors $e_i$ ($i=1,\cdots,n$) of $\mathbb{S}^{N-1}$, define
\begin{equation}\label{subsolution}
\underline{V}(t,x):=\max\limits_{1\leq i\leq n} \{U_{e_i}(x\cdot e_i-c_{e_i}t,x)\},
\end{equation}
which is a subsolution of  \eqref{1.1} by the comparison principle for parabolic equations. Our first result shows the existence of a curved front which converges to pulating fronts along its asymptotic planes under some sufficient conditions. This curved front is also a transition front with unchanging interfaces as $t$ varies.

\begin{theorem}\label{thm-existence}
For any $\{e_i\}_{i=1}^n$ of $\mathbb{S}^{N-1}$ such that
\begin{itemize}
\item[(i)] $e_i\cdot e_0>0$ for a fixed $e_0\in\mathbb{S}^{N-1}$ and $e_i\neq e_j$ for $i\neq j$,

\item[(ii)]  $\hat{c}:=g(e_i)\equiv \hbox{constant}$ for any $i\in\{1,\cdots,n\}$, where $g(x):=c_{\frac{x}{|x|}}/(\frac{x}{|x|}\cdot e_0)$ for $x\in \mathbb{R}^N \setminus \{0\}$,

\item[(iii)] $\hat{c}>g(e)$   for $e\in \mathcal{L}(\mathcal{Q})\setminus\{e_i\}_{i=1}^n$,  where  $\mathcal{L}(\mathcal{Q}):=\{e\in\mathbb{S}^{N-1}; x\cdot e\geq 0\text{ for all }x\in \mathcal{Q}\}$,

\item[(iv)] $\nabla g(e_i)\cdot e_j<0$ for every $i\neq j$,
\end{itemize}
there exists a transition front $V(t,x)$ of \eqref{1.1} with  $\Gamma_t$, $\Omega_t^{\pm}$ given by
\be\label{TF-sets} 
\Gamma_t=\partial\mathcal{Q}+\hat{c} t e_0,\quad \Omega_t^-=\mathcal{Q}+\hat{c} t e_0,\quad \Omega_t^+=\R^N\setminus \overline{\mathcal{Q}}  +\hat{c} t e_0,
\ee
satisfying $V_t(t,x)>0$ for all $(t,x)\in\mathbb{R}\times\mathbb{R}^N$ and
\begin{equation}\label{thm-existence-1}
|V(t,x)-\underline{V}(t,x)|\to 0,\text{ uniformly as }d((t,x),\mathcal{R}+\hat{c}t e_0)\to +\infty.
\end{equation}
Moreover, such a transition front is unique in the sense that if an entire solution $\widetilde{V}(t,x)$ satisfies \eqref{thm-existence-1}, then $\widetilde{V}(t,x)\equiv V(t,x)$.
\end{theorem}

Here are some comments on conditions (i)-(iv) of Theorem~\ref{thm-existence}. The condition (i) ensures that the solution $V(t,x)$ is a transition front propagating in the direction $e_0$. The condition (ii) implies that pulsating fronts $U_{e_i}(x\cdot e_i -c_{e_i} t,x)$ have the same speed $\hat{c}=c_{e_i}/(e_i\cdot e_0)$ in the direction $e_0$. This is why they can glue together. The condition (iii) ensures that the interfaces of $V(t,x)$ can keep the shape of $\partial\mathcal{Q}$ as $t$ varies. We can actually show that (ii)-(iii) are necessary conditions for the existence of curved fronts satisfying \eqref{thm-existence-1} by the same strategy as for the 2-dimensional case, see \cite{MR4334733}. The condition (iv) is a technical assumption. We are not sure whether the conclusion is true or not when (iv) is losing. We will prove later that $c_e$ is differentiable with respect to $e\in\mathbb{S}^{N-1}$ and therefore, one can deduce that $g(x)$ is differentiable at $e_i$ with respect to $x\in\mathbb{R}^N\setminus\{0\}$. So that, (iv) is well-defined. Furthermore, the set of admissible speeds $c_e$ is shown to be rather large, and it is even conjectured to be a set of all continuous sign-unchanging functions with respect to $e\in\mathbb{S}^{N-1}$ in \cite{MR4288217}. This means that (iv)  could be easily satisfied.

In homogeneous case, the speed $c_f$ of planar front is uniquely determined by the reaction term $f$ and invariant with respect to the direction. Thus, if $e_i\cdot e_0$ is equivalent to a positive constant  for all $i\in\{1,\dots,n\}$, then conditions (ii)-(iv) are automatically satisfied. In this case, the curved front is the pyramidal front in \cite{MR2318388,MR2494701}. If we take $N=2$, $e_0=(0,1)$, $e_1=(\cos\alpha,\sin\alpha)$, $e_2=(\cos\beta,\sin\beta)$ with $0<\alpha<\beta<\pi$ in Theorem~\ref{thm-existence}, then it turns to the curved front in dimension $2$ and it is not difficult to show that conditions (i)-(iv) are equivalent to those in \cite[Thorem~1.2]{MR4334733}. 

%

Next, we show that the curved front $V(t,x)$ is  asymptotically stable.

\begin{theorem}\label{thm-stability}
Assume that (i)-(iv) of Theorem~\ref{thm-existence} hold. For the Cauchy problem of equation (\ref{1.1}), that is,
\begin{equation}\label{Chaucy-problem}
\left\{
\begin{aligned}
&u_t-\nabla\cdot(A(x)\nabla u)=f(x,u), &&t>0,x\in\mathbb{R}^N,\\
&u(0,x)=u_0(x),  && x\in\mathbb{R}^N,
\end{aligned}
\right.
\end{equation}
if the initial value $u_0(x)$ satisfies $0\leq u_0(x)\leq 1$ and
\begin{equation}\label{Chaucy-problem-initialvalue}
|u_0(x)-\underline{V}(0,x)|\to 0,\text{ uniformly as }d(x,\mathcal{R})\to +\infty,
\end{equation}
then the solution $u(t,x)$ of \eqref{Chaucy-problem} satisfies
\begin{equation*}
\lim\limits_{t\to+\infty}||u(t,x)-V(t,x)||_{L^{\infty}(\mathbb{R}^N)}=0.
\end{equation*}
\end{theorem}

Finally, we point out that if the inequalities in (iii)-(iv) of Theorem~\ref{thm-existence} are reversed, one can still construct transition fronts of \eqref{1.1} propagating in the direction $e_0$ with speed $\hat{c}$, but the interfaces $\Gamma_t$ are  $-\partial \mathcal{Q}+\hat{c} t e_0$. 

\begin{theorem}\label{coro}
For any n vectors $\{e_i\}_{i=1}^n$ of $\mathbb{S}^{N-1}$ satisfying
\begin{itemize}
\item[(i)] $e_i\cdot e_0>0$ for a fixed $e_0\in\mathbb{S}^{N-1}$ and $e_i\neq e_j$ for $i\neq j$,

\item[(ii)]  $\hat{c}:=g(e_i)\equiv \hbox{constant}$ for any $i\in\{1,\cdots,n\}$,

\item[(iii)] $\hat{c}<g(e)$   for $e\in \mathcal{L}(\mathcal{Q})\setminus\{e_i\}_{i=1}^n$, 

\item[(iv)] $\nabla g(e_i)\cdot e_j>0$ for every $i\neq j$,
\end{itemize}
there exists a unique transition front $W(t,x)$ of \eqref{1.1} with  $\Gamma_t$, $\Omega_t^{\pm}$ given by
\be\label{TF-sets2} 
\Gamma_t=-\partial\mathcal{Q}+\hat{c} t e_0,\quad  \Omega_t^+=-\mathcal{Q}+\hat{c} t e_0,\quad \Omega_t^-=-\R^N\setminus \overline{\mathcal{Q}}  +\hat{c} t e_0,
\ee
 satisfying $W_t(t,x)>0$ for all $(t,x)\in\mathbb{R}\times\mathbb{R}^N$ and
\begin{equation}\label{coro-1}
	|W(t,x)-\overline{W}(t,x)|\to 0,\text{ uniformly as }d((t,x),-\mathcal{R}+\hat{c}t e_0)\to +\infty,
\end{equation}
where $\overline{W}(t,x):=\min\limits_{1\leq i\leq n} \{U_{e_i}(x\cdot e_i-c_{e_i}t,x)\}$, which is a supersolution of \eqref{1.1}.
\end{theorem}

The stability of the curved front $W(t,x)$ can also be deduced. 


\begin{theorem}\label{coro-stability}
Assume that (i)-(iv) of Theorem~\ref{coro} hold. For the Cauchy problem \eqref{Chaucy-problem}, if the initial value $u_0(x)$ satisfies $0\leq u_0(x)\leq 1$ and
\begin{equation}\label{Chaucy-problem-initialvalue2}
|u_0(x)-\underline{W}(0,x)|\to 0,\text{ uniformly as }d(x,-\mathcal{R})\to +\infty,
\end{equation}
then the solution $u(t,x)$ of \eqref{Chaucy-problem} satisfies
\begin{equation*}
\lim\limits_{t\to+\infty}||u(t,x)-W(t,x)||_{L^{\infty}(\mathbb{R}^N)}=0.
\end{equation*}
\end{theorem}

\begin{figure}[H]
	\centering
\includegraphics[scale=0.3]{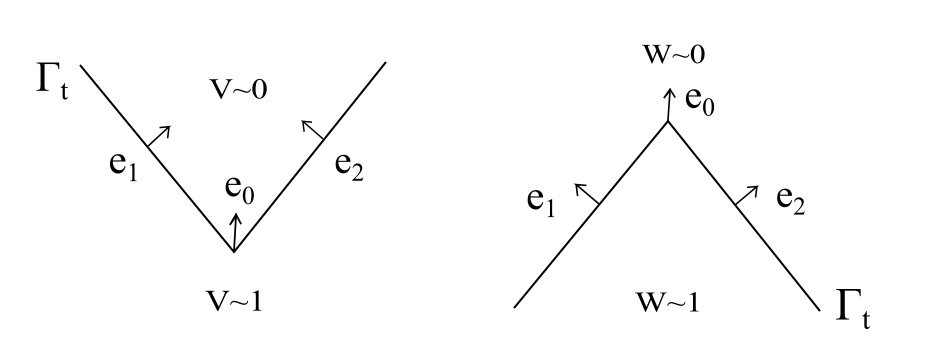}
	\caption{2-dimensional curved fronts in Theorem \ref{thm-existence} and Theorem \ref{coro}.}
\end{figure}
One can see Figure~1 for curved fronts in Theorem \ref{thm-existence} and Theorem \ref{coro} in dimension 2.
Obviously, for any fixed set of vectors $\{e_i\}_{i=1}^n$ of $\mathbb{S}^{N-1}$, the curved fronts in Theorem \ref{thm-existence} and Theorem \ref{coro} can not exist simultaneously. In the homogeneous case, the conditions of Theorem \ref{coro} can not be satisfied and hence such a kind of curved fronts can not exist. 

The rest of this paper is orgnized as follows. In Section 2, we list some preliminaries and construct sub and supersolutions by mixing moving pulsating fronts. In Section 3, we show the existence and uniqueness of the curved front in Theorem \ref{thm-existence} and Theorem \ref{coro}. Section 4 is devoted to the proof of the stability, that is, Theorem \ref{thm-stability} and Theorem \ref{coro-stability}. Lastly, in the appendix we prove  lemmas in Section~2.

\section{Preliminaries}

The main goal of this section is to construct appropriate sub and supersolutions of \eqref{1.1} by pulsating fronts. We will then need some properties of pulsating fronts, especially the differentiability with respect to the direction. We also prove a  comparison principle between sub and supersolutions in the next subsection. In Subsection~\ref{subs:surface}, we introduce a surface with asymptotic planes.  In Subsection~\ref{subs:subsup}, we construct sub and supersolutions. 

In the sequel of this paper, we will first consider the case $e_0=(0,0,\cdots,1)$ and write \eqref{1.1} in the following form for convenience
\begin{equation}\label{initial-equation}
	u_t-\nabla_{x,y}\cdot(A(x,y)\nabla_{x,y} u)=f(x,y,u),\quad (t,x,y)\in \mathbb{R}\times \mathbb{R}^{N-1}\times\mathbb{R},
\end{equation}
where $\nabla_{x,y}=\nabla_x +\partial_y$.  For general $e_0\in \mathbb{S}^{N-1}$, the proofs of Theorems \ref{thm-existence}, \ref{thm-stability} and \ref{coro} are similar as for $e_0=(0,0,\cdots,1)$ and we will discuss them in the detailed proofs. 

\subsection{Properties of pulsating fronts and a comparison principle}
Firstly, by substituting the pulsating front $U_e((x,y)\cdot e-c_et,x,y)$ into \eqref{initial-equation}, one gets that $(U_e(\xi,x,y),c_e)$ satisfies
\begin{equation}\label{property-1}
	c_e\partial_{\xi}U_e+eAe\partial_{\xi\xi}U_e+ eA\nabla_{x,y} \partial_{\xi}U_e+\nabla_{x,y} \cdot(Ae\partial_{\xi}U_e)+\nabla_{x,y} \cdot(A\nabla_{x,y} U_e)+f(x, U_e)=0,
\end{equation}
for each $e\in\mathbb{S}^{N-1}$ and $(\xi,x,y)\in\mathbb{R}\times\mathbb{T}^N$.

By Lemmas 2.4-2.5 in \cite{MR4334733}, which can be trivially extended to the $N$-dimensional space, we have following properties of pulsating fronts.

\begin{lemma}\label{pulsating-lemma-1}
	For any pulsating front $(U_e(\xi,x,y),c_e)$ with $c_e>0$, there exist $\mu_1>0$ and $C_1>0$ independent of $e$ such that
	\begin{equation}
		|\partial_{\xi}U_e(\xi,x,y)|,\, |\partial_{\xi\xi}U_e(\xi,x,y)|,\, |\nabla\partial_{\xi}U_e(\xi,x,y)|\leq C_1 e^{-\mu_1|\xi|},
	\end{equation}
	for $(\xi,x,y)\in\mathbb{R}\times\mathbb{T}^N$.
\end{lemma}

\begin{lemma}\label{pulsating-lemma-2}
	For any $R>0$, there are $\delta\in(0,\frac{1}{2})$ and $r>0$ independent of $e$ such that
	\begin{equation*}
		\delta\leq U_e(\xi,x,y)\leq 1-\delta\text{ and }-\partial_{\xi}U_e(\xi,x,y)\geq r
	\end{equation*}
	for $-R\leq\xi\leq R$ and $(x,y)\in\mathbb{T}^N$.
\end{lemma}

Next, we show the continuous differentiability and the boundedness of derivatives of $U_e$ and $c_e$ with respect to $e$. The proof is inspired from \cite{MR3689331,MR3772873} and we postpone it in the appendix.

\begin{lemma}\label{pulsating-lemma-3}
	For any $b\in\mathbb{R}^N\setminus\{0\}$, define $U_b$ and $c_b$ as follows:
	$$U_b=U_{\frac{b}{|b|}}\text{ and }c_b=c_{\frac{b}{|b|}}.$$
	Then, $U_b$ and $c_b$ are doubly continuously Fr\'{e}chet differentiable in $\mathbb{R}^N\setminus\{0\}$. Precisely, there exist linear operators $(U_b',c_b'):\mathbb{R}^N\to L^2(\mathbb{R}\times\mathbb{T}^N)\times\mathbb{R}$ and $(U_b'',c_b''):\mathbb{R}^N\times\mathbb{R}^N\to L^2(\mathbb{R}\times\mathbb{T}^N)\times\mathbb{R}$ such that for any $\tau, \rho\in\mathbb{R}^N$,
	\begin{equation*}
		\begin{aligned}
			& (U_{b+\tau},c_{b+\tau})-(U_b,c_b)=(U_b',c_b')\cdot\tau+o(|\tau|),\\
			& (U_{b+\rho}'\cdot\tau,c_{b+\rho}'\cdot\tau)-(U_b'\cdot\tau,c_b'\cdot\tau)=(U_b''\cdot\tau,c_b''\cdot\tau)\cdot\rho+o(|\rho|),
		\end{aligned}
	\end{equation*}
	as $|\tau|, |\rho|\to 0$. In addition, for any $e\in\mathbb{S}^{N-1}$, the Fr\'{e}chet derivatives up to second order of $U_e$ and $c_e$ with respect to $e$ are all bounded in the sense that
	\begin{equation*}
		\begin{aligned}
			& ||U_e'||=\sup_{|\tau|=1}||U_e'\cdot\tau||_{L^{\infty}(\mathbb{R}\times\mathbb{T}^N)}<+\infty,\quad 
			||U_e''||=\sup_{|\tau|=1,|\rho|=1} ||(U_e''\cdot\tau)\cdot\rho||_{L^{\infty}(\mathbb{R}\times\mathbb{T}^N)}<+\infty,\\
			&||\partial_{\xi}U_e'||=\sup_{|\tau|=1}||\partial_{\xi}U_e'\cdot\tau||_{L^{\infty}(\mathbb{R}\times\mathbb{T}^N)}<+\infty,\quad 
			||\nabla U_e'||=\sup_{|\tau|=1}||\nabla U_e'\cdot\tau||_{L^{\infty}(\mathbb{R}\times\mathbb{T}^N)}<+\infty,\\
			& ||c_e'||=\sup_{|\tau|=1}|c_e'\cdot\tau|<+\infty,\quad 
			||c_e''||=\sup_{|\tau|=1,|\rho|=1}|(c_e''\cdot\tau)\cdot\rho|<+\infty.
		\end{aligned}
	\end{equation*}
	Furthermore, for any $e\in\mathbb{S}^{N-1}$, there exist $\mu_2>0$ and $C_2>0$ independent of $e$ such that
	\begin{equation}\label{pulsating-lemma-4}
		|(U_e'\cdot\tau)(\xi,x,y)|,\, |(\partial_{\xi}U_e'\cdot\tau)(\xi,x,y)|\leq C_2e^{-\mu_2|\xi|}|\tau|,
	\end{equation}
	for all $\tau\in\mathbb{R}^N$ and $(\xi,x,y)\in\mathbb{R}\times\mathbb{T}^N$.
\end{lemma}

Finally,  we prove a comparison principle for sub and supersolutions. This can be regarded as a generalization of \cite[Theorem 1.12]{MR2898886}. We will prove it in the appendix.

\begin{definition}
By a sub-invasion (resp. sup-invasion) $u$ of $0$ by $1$,	we mean that $u\in [0,1]$ is a subsolution  (resp. supsolution) of \eqref{1.1} satisfying \eqref{transition4} and
	\begin{equation*}
		\Omega_{t}^+\supset\Omega_{s}^+ \text{ for all }t\geq s,
	\end{equation*}
	and
	\begin{equation*}
		d(\Gamma_t,\Gamma_s)\to+\infty\text{ as }|t-s|\to+\infty.
	\end{equation*}
\end{definition}

\begin{lemma}\label{pulsating-lemma-5}
	Let $u(t,x)$ be a sub-invasion of $0$ by $1$ of \eqref{1.1} associated to the families $(\Omega_{t}^{\pm})_{t\in\mathbb{R}}$ and $(\Gamma_t)_{t\in\mathbb{R}}$ and $\widetilde{u}(t,x)$ be a super-invasion of $0$ by $1$ of \eqref{1.1} with sets $(\widetilde{\Omega}_{t}^{\pm})_{t\in\mathbb{R}}$ and $(\widetilde{\Gamma}_t)_{t\in\mathbb{R}}$. 
	If $\widetilde{\Omega}_{t}^{-}\subset \Omega_{t}^{-}$ for all $t\in\mathbb{R}$, then there exists a smallest $T\in\mathbb{R}$ such that
	\begin{equation*}
		\widetilde{u}(t+T,x)\geq u(t,x)\text{ for all }(t,x)\in\mathbb{R}\times\mathbb{R}^N.
	\end{equation*}
	Furthermore, there exists a sequence $\{(t_n,x_n)\}_{n\in\mathbb{N}}$ in $\mathbb{R}\times\mathbb{R}^N$ such that
	\begin{equation*}
		\{d(x_n,\Gamma_{t_n})\}_{n\in\mathbb{N}}\text{ is bounded and }\widetilde{u}(t_n+T,x_n)-u(t_n,x_n)\to 0\text{ as }n\to+\infty.
	\end{equation*}
\end{lemma}

\subsection{A surface with asymptotic planes}\label{subs:surface}

Take $n$ unit vectors $\nu_i\in\mathbb{S}^{N-2}$ and angles $\theta_i\in (0,\pi/2]$ ($i=1,\cdots,n$) such that $(\nu_i,\theta_i)\neq (\nu_j,\theta_j)$ for any $i\neq j$. Let $e_i=(\nu_i\cos\theta_i,\sin\theta_i)$ for $i=1,\cdots,n$. Notice that $e_i\cdot e_0>0$ for all $i$. Recall from Section~1 that $Q_i$ is the hyperplane determined by $e_i$, that is,
$$Q_i=\{(x,y)\in\mathbb{R}^N; x\cdot\nu_i\cos\theta_i+y\sin\theta_i=0\},$$
and $\mathcal{Q}$ is the unbounded polytope enclosed by $Q_1,Q_2,\cdots,Q_n$.
Moreover, $\partial\mathcal{Q}$,  $\widetilde{Q}_i$, $R_{ij}$ and $\mathcal{R}$ are the boundary, facets, ridges and the set of all ridges of $\mathcal{Q}$ respectively.  We define $\widehat{\mathcal{R}}$ as the projection of $\mathcal{R}$ on the $x$-plane, that is,
$$\widehat{\mathcal{R}}:=\{x\in\mathbb{R}^{N-1};\text{ there exists $y$ such that }(x,y)\in\mathcal{R}\}.$$
Similarly, let $\widehat{Q}_i$ be the projection of $\widetilde{Q}_i$ on the $x$-plane.

By \cite[Section~2.1]{preprint1}, we know that there is a smooth convex surface $y=\varphi(x)$ in $\mathcal{Q}$ such that the surface is approaching the facets $\widetilde{Q}_i$ exponentially as the surface diverging away from the ridges. The surface $y=\varphi(x)$ is determined by the equation
\begin{equation}\label{surface}
\sum_{i=1}^{n}e^{-q_i(x,y)}=1,
\end{equation}
where $q_i(x,y):=x\cdot\nu_i\cos\theta_i+y\sin\theta_i$. Let $\hat{q}_i(x)=q_i(x,\varphi(x))$ for every $i$ and let
$$h(x)=\sum_{j,k\in\{1,\cdots,n\},j\neq k}e^{-(\hat{q}_j(x)+\hat{q}_k(x))}, \hbox{ for $x\in\R^{N-1}$}.$$
Then, $h(x)$ is decaying exponentially in every $\widehat{Q}_i$ as $d(x,\widehat{\mathcal{R}})\rightarrow +\infty$. Precisely, by \cite[Section~2.1]{preprint1}, $\hat{q}_i(x)\to0$ and $\hat{q}_j(x)\to+\infty$ for all $j\neq i$ and $x\in\widehat{Q}_i$ as ${\rm dist}(x,\widehat{\mathcal{R}})\to+\infty$. Then, the surface $y=\varphi(x)$ has the following properties.

\begin{lemma}[\cite{preprint1}]\label{surface-lemma}
There exists $C_3>0$ such that for every $i\in\{1,\cdots,n\}$,
\begin{equation}\label{esti-surface-2}
|\nabla\varphi(x)+\nu_i \cot\theta_i|\leq C_3 h(x),\quad \hbox{ for $x\in\widehat{Q}_i$}.
\end{equation}
 and
 \begin{equation}\label{esti-surface-3}
 |\nabla^2\varphi(x)|,\, |\nabla^3\varphi(x)|\leq C_3 h(x),\quad \hbox{ for $x\in\mathbb{R}^{N-1}$}.
 \end{equation}
for $x\in\mathbb{R}^{N-1}$.
\end{lemma}

The surface $y=\varphi(x)$ is used to construct supersolutions. In the sequel, we will also consider $y=-\varphi(-x)$ which is in $-\mathcal{Q}$ and approaching $-\partial\mathcal{Q}$ exponentially,  to construct subsolutions. 

\subsection{Supersolutions and subsolutions}\label{subs:subsup}
Take $e_0=(0,0,\cdots,1)$ and $\{e_i\}_{i=1}^n$ such that $e_i\cdot e_0>0$ for all $i$. Let $y=\varphi(x)$ be the surface determined by $\{e_i\}_{i=1}^n$ given as in Section~\ref{subs:surface} and rescale it by the parameter $\alpha>0$, that is, $y=\varphi(\alpha x)/\alpha$. For every point $(x,y)$ on the surface $y=\varphi(\alpha x)/\alpha$, there is a  unit normal
\begin{equation}\label{e(x)}
e(x)=\left(-\frac{\nabla\varphi(\alpha x)}{\sqrt{1+|\nabla\varphi(\alpha x)|^2}},\frac{1}{\sqrt{1+|\nabla\varphi(\alpha x)|^2}}\right).
\end{equation}
Direct computation yields that for all $i,j\in\{1,\cdots,N-1\}$,
\begin{equation*}
		 \partial_{x_i} e(x)=\left(-\frac{\alpha\nabla\partial_{x_i}\varphi}{(1+|\nabla\varphi|^2)^{\frac{3}{2}}},-\frac{\alpha\nabla\varphi\nabla\partial_{x_i}\varphi}{(1+|\nabla\varphi|^2)^{\frac{3}{2}}}\right),
		 \end{equation*}
		 and
\begin{equation*}
\begin{split}
\partial_{x_ix_j} e(x)=&\Big(-\frac{\alpha^2\nabla\partial_{x_ix_j}\varphi}{(1+|\nabla\varphi|^2)^{\frac{3}{2}}}+\frac{3\alpha^2\nabla\varphi\nabla\partial_{x_i}\varphi\nabla\partial_{x_j}\varphi}{(1+|\nabla\varphi|^2)^{\frac{5}{2}}},\\
&-\frac{\alpha^2\partial_{x_j}(\nabla\varphi\nabla\partial_{x_i}\varphi)}{(1+|\nabla\varphi|^2)^{\frac{3}{2}}} +\frac{3\alpha^2|\nabla\varphi|^2\nabla\partial_{x_i}\varphi\nabla\partial_{x_j}\varphi}{(1+|\nabla\varphi|^2)^{\frac{5}{2}}}\Big),	
\end{split}
\end{equation*}
where $\varphi$ is taking value at $\alpha x$.
Then it follows from (\ref{esti-surface-3}) that there are positive constants $M_1$ and $M_2$ such that
\begin{equation}\label{esti-e(x)}
|\partial_{x_i} e(x)|\leq\alpha M_1 h(\alpha x) \text{ and }|\partial_{x_ix_j} e(x)|\leq\alpha^2 M_2 h(\alpha x),
\end{equation}
 for all $x\in\mathbb{R}^{N-1}$ and $i,j\in\{1,\cdots,N-1\}$.
For $\varepsilon>0$ and $\alpha>0$, define
\begin{equation}\label{eq:supersolution}
\overline{V}(t,x,y):=\min\{U_{e(x)}(\overline{\xi}(t,x,y),x,y)+\varepsilon h(\alpha x),1\},
\end{equation}
 and 
 \begin{equation}\label{eq:subsolution}
 \underline{W}(t,x,y):=\max\{U_{e(-x)}(\underline{\xi}(t,x,y),x,y)-\varepsilon h(-\alpha x),0\},
\end{equation}
for $(t,x,y)\in\R\times\R^N$, where
\begin{equation}\label{uxilxi}
\overline{\xi}(t,x,y)=\frac{y-\overline{c} t-\varphi(\alpha x)/\alpha}{\sqrt{1+|\nabla\varphi(\alpha x)|^2}}, \quad
\underline{\xi}(t,x,y)=\frac{y-\overline{c} t+\varphi(-\alpha x)/\alpha}{\sqrt{1+|\nabla\varphi(-\alpha x)|^2}},
\end{equation}
and $\overline{c}$ is a positive constant.
We show that $\overline{V}$ and $\underline{W}$ are a supersolution and a subsolution respectively of (\ref{initial-equation}) for small $\varepsilon$ and $\alpha$ under some conditions.

\begin{lemma}\label{supersolution}
If it holds that 
\begin{equation}\label{claim-super}
	-\overline{\xi}_t-c_{e(x)}\geq Ch(\alpha x),\quad \text{ for some }C>0,
\end{equation}
then there exist $\varepsilon_0>0$ and $\alpha(\varepsilon_0)>0$ such that for $0<\varepsilon<\varepsilon_0$ and $0<\alpha<\alpha(\varepsilon_0)$, 
 the function $\overline{V}(t,x,y)$ is a supersolution of (\ref{initial-equation}).

If it holds that 
\begin{equation}\label{claim-sub}
	\underline{\xi}_t+c_{e(-x)}\ge Ch(-\alpha x),\quad \text{ for some }C>0,
\end{equation}
then there exist $\varepsilon_0>0$ and $\alpha(\varepsilon_0)>0$ such that for $0<\varepsilon<\varepsilon_0$ and $0<\alpha<\alpha(\varepsilon_0)$,  the function $\underline{W}(t,x,y)$ is a subsolution of (\ref{initial-equation}).
\end{lemma}

\begin{proof}
By the comparison principle, we only have to deal with the scenarios $0<\overline{V}<1$ and $0<\underline{W}<1$.

Case~1: $\overline{V}$ is a supersolution. By the definition of $\overline{V}$, the boundedness of $a_{ij}(x,y)$ and its derivatives$(1\leq i,j\leq N)$, and (\ref{property-1}), one can compute that
\begin{eqnarray}\label{compute-supersolu-1}
\begin{array}{lll}
N\overline{V}&:=&\overline{V}_t-\nabla_{x,y}\cdot(A(x,y)\nabla_{x,y}\overline{V})-f(x,y,\overline{V})\\
&\geq& \partial_{\xi}U_{e(x)}(\overline{\xi}_t+c_{e(x)})
-\sum_{i,j=1}^{N-1}a_{ij}U_{e(x)}''\cdot \partial_{x_i}e(x)\cdot\partial_{x_j}e(x)\\
&&-2\sum_{i,j=1}^{N-1}a_{ij}\partial_{x_j}\overline{\xi}\partial_{\xi}U_{e(x)}'\cdot\partial_{x_i}e(x) 
-2\sum_{i,j=1}^{N-1}a_{ij}\partial_{x_j}U_{e(x)}'\cdot \partial_{x_i}e(x)\\
&&-\sum_{i,j=1}^{N-1}a_{ij}U_{e(x)}'\cdot \partial_{x_ix_j}e(x) -\sum_{i=1}^{N-1}||A||_{C^1(\mathbb{R}^N)}U_{e(x)}'\cdot\partial_{x_i}e(x)\\
&&-N^2||A||_{C^1(\mathbb{R}^N)}|\partial_{\xi}U_{e(x)}(\nabla_{x,y}\overline{\xi}-e(x))| -N^2||A||_{C^0(\mathbb{R}^N)}|\partial_{\xi}U_{e(x)}\nabla^2_{x,y}\overline{\xi}|\\
&&-\partial_{\xi\xi}U_{e(x)}(\nabla_{x,y}\overline{\xi} A\nabla_{x,y}\overline{\xi}-e(x)Ae(x))-2\nabla_{x,y}\partial_{\xi}U_{e(x)}A(\nabla_{x,y}\overline{\xi}-e(x))\\
&& -\varepsilon\nabla_{x,y}\cdot(A(\nabla_xh(\alpha x),0))+f(x,y,U_{e(x)})-f(x,y,\overline{V}),
	\end{array}
\end{eqnarray}
where $\partial_{\xi}U_{e(x)}$, $\partial_{\xi\xi}U_{e(x)}$, $\nabla_{x,y}\partial_{\xi}U_{e(x)}$, $
U_{e(x)}''\cdot \partial_{x_i}e(x)\cdot\partial_{x_j}e(x)$, $
\partial_{\xi}U_{e(x)}'\cdot\partial_{x_i}e(x)$, $
\partial_{x_j}U_{e(x)}'\cdot \partial_{x_i}e(x)$, $ U_{e(x)}'\cdot \partial_{x_ix_j}e(x)$, $U_{e(x)}'\cdot\partial_{x_i}e(x)$ are taking values at $(\overline{\xi}(t,x,y),x,y)$, $A$ is taking values at $(x,y)$ and $\overline{V}$, $\overline{\xi}$, $\overline{\xi}_t$, $\nabla_{x,y}\overline{\xi}$, $\nabla^2_{x,y}\overline{\xi}$ are taking values at $(t,x,y)$. We compute the derivatives of $\overline{\xi}(t,x,y)$ as follows: 
 \begin{align*}
&\overline{\xi}_t = \displaystyle-\frac{\overline{c}}{\sqrt{1+|\nabla\varphi|^2}},\\ \allowdisplaybreaks
&\partial_y\overline{\xi}  =\displaystyle\frac{1}{\sqrt{1+|\nabla\varphi|^2}},\\
& \nabla_x\overline{\xi} =\displaystyle-\frac{\nabla\varphi}{\sqrt{1+|\nabla\varphi|^2}}-\alpha\frac{\nabla\varphi\cdot\nabla^2\varphi}{1+|\nabla\varphi|^2}\overline{\xi},\\
& \nabla_{x,y}\overline{\xi}-e(x) =\displaystyle\left(-\alpha\frac{\nabla\varphi\cdot\nabla^2\varphi}{1+|\nabla\varphi|^2}\overline{\xi},\text{ }0\right),\\
 &\nabla_{x,y}\overline{\xi}+e(x) =\displaystyle\left(-2\frac{\nabla\varphi}{\sqrt{1+|\nabla\varphi|^2}}-\alpha\frac{\nabla\varphi\cdot\nabla^2\varphi}{1+|\nabla\varphi|^2}\overline{\xi},\text{ }\frac{2}{\sqrt{1+|\nabla\varphi|^2}}\right),\\
 & \nabla^2_{x,y}\overline{\xi} =\displaystyle-\alpha\frac{\nabla^2\varphi}{\sqrt{1+|\nabla\varphi|^2}}+2\alpha\frac{\nabla^2\varphi\cdot\nabla^2\varphi}{(1+|\nabla\varphi|^2)^{\frac{3}{2}}}+3\alpha^2\frac{(\nabla\varphi\cdot\nabla^2\varphi)^2}{(1+|\nabla\varphi|^2)^2}\overline{\xi}-\alpha^2\frac{\nabla\cdot(\nabla\varphi\cdot\nabla^2\varphi)}{1+|\nabla\varphi|^2}\overline{\xi},
 \end{align*}
where $\varphi$ is taking value at $\alpha x$ and $\overline{\xi}$ is taking value at $(t,x,y)$. By (\ref{esti-surface-3}) and the boundedness of $|\nabla\varphi(\alpha x)|$ and $||A||_{C^1(\mathbb{R}^N)}$, there is $C_4>0$ such that $|\nabla^2_{x,y}\overline{\xi}|\leq \alpha C_4(1+\alpha|\overline{\xi}|)h(\alpha x)$, $|\nabla_{x,y}\overline{\xi}-e(x)|\leq \alpha C_4|\overline{\xi}|h(\alpha x)$ and $|\nabla_{x,y}\overline{\xi}+e(x)|\leq \alpha C_4(1+|\overline{\xi}|)h(\alpha x)$ for any $\overline{\xi}\in\mathbb{R}$. We also know from Lemma \ref{pulsating-lemma-1} that $|\partial_{\xi}U_{e(x)}\overline{\xi}|$, $|\partial_{\xi\xi}U_{e(x)}\overline{\xi}|$, $|\partial_{\xi\xi}U_{e(x)}\overline{\xi}^2|$ and $|\nabla_{x,y}\partial_{\xi}U_{e(x)}\overline{\xi}|$ are bounded uniformly for $(\overline{\xi},x,y)\in\mathbb{R}\times\mathbb{T}^N$. Thus, by noticing that $\nabla_{x,y}\overline{\xi} A\nabla_{x,y}\overline{\xi}-e(x)Ae(x)=(\nabla_{x,y}\overline{\xi}-e(x))A(\nabla_{x,y}\overline{\xi}+e(x))$, there is $C_5>0$ such that
\begin{equation}\label{compute-supersolu-2}
	\begin{aligned}
& n^2||A||_{C^1(\mathbb{R}^N)}|\partial_{\xi}U_{e(x)}(\nabla_{x,y}\overline{\xi}-e(x))|+n^2||A||_{C^0(\mathbb{R}^N)}|\partial_{\xi}U_{e(x)}\nabla^2_{x,y}\overline{\xi}|\\
& +|\partial_{\xi\xi}U_{e(x)}(\nabla_{x,y}\overline{\xi} A\nabla_{x,y}\overline{\xi}-e(x)Ae(x))|+|2\nabla_{x,y}\partial_{\xi}U_{e(x)}A(\nabla_{x,y}\overline{\xi}-e(x))|\leq \alpha C_5h(\alpha x).
	\end{aligned}
\end{equation}
By Lemma \ref{pulsating-lemma-3}, one knows that $||U_{e(x)}'||,||U_{e(x)}''||,||\partial_{\xi}U_{e(x)}'||$ and $||\nabla_xU_{e(x)}'||$ are all bounded for $x\in\mathbb{R}^{N-1}$. Then by (\ref{pulsating-lemma-4}), (\ref{esti-e(x)}), there is $C_6>0$ such that
\begin{equation}\label{compute-supersolu-3}
\begin{aligned}
	&
	\left|\sum_{i,j=1}^{N-1}a_{ij}U_{e(x)}''\cdot \partial_{x_i}e(x)\cdot\partial_{x_j}e(x)\right|
	+\left|2\sum_{i,j=1}^{N-1}a_{ij}\partial_{x_j}\overline{\xi}\partial_{\xi}U_{e(x)}'\cdot\partial_{x_i}e(x) \right|
	+\left|2\sum_{i,j=1}^{N-1}a_{ij}\partial_{x_j}U_{e(x)}'\cdot \partial_{x_i}e(x)\right|\\
	&+\left||\sum_{i,j=1}^{N-1}a_{ij}U_{e(x)}'\cdot \partial_{x_ix_j}e(x)|\right|
	+\left|\sum_{i=1}^{N-1}||A||_{C^1(\mathbb{R}^N)}U_{e(x)}'\cdot\partial_{x_i}e(x)\right|\leq \alpha C_6h(\alpha x).
\end{aligned}
\end{equation}
A direct computation on $h(x)$ gives that there is $C_7>0$ such that
\begin{equation}\label{compute-supersolu-4}
	\begin{aligned}
&|\nabla_{x,y}\cdot(A(\nabla_xh(\alpha x),0))|
\leq \alpha N^2||A||_{C^1(\mathbb{R}^N)}|\nabla_xh(\alpha x)|+\alpha^2 N^2||A||_{C^0(\mathbb{R}^N)}|\nabla^2_xh(\alpha x)|\\
&=\alpha N^2||A||_{C^1(\mathbb{R}^N)}|\sum_{j,k\in\{1,\cdots,n\},j\neq k}e^{-(\hat{q}_j(\alpha x)+\hat{q}_k(\alpha x))}|\nu_i\cos\theta_i+\nu_j\cos\theta_j+\nabla\varphi(\alpha x)(\sin\theta_i+\sin\theta_j)|\\
&+\alpha^2 N^2||A||_{C^0(\mathbb{R}^N)}\sum_{j,k\in\{1,\cdots,n\},j\neq k}e^{-(\hat{q}_j(\alpha x)+\hat{q}_k(\alpha x))}|\nu_i\cos\theta_i+\nu_j\cos\theta_j+\nabla\varphi(\alpha x)(\sin\theta_i+\sin\theta_j)|^2\\
& +\alpha^2 N^2||A||_{C^0(\mathbb{R}^N)}\left|\sum_{j,k\in\{1,\cdots,n\},j\neq k}e^{-(\hat{q}_j(\alpha x)+\hat{q}_k(\alpha x))}\nabla^2\varphi(\alpha x)(\sin\theta_i+\sin\theta_j)\right|\\
 &\leq \alpha^2C_7h(\alpha x).
    \end{aligned}
\end{equation}
Noticing that $\partial_{\xi}U_{e(x)}<0$, it follows from (\ref{claim-super}), (\ref{compute-supersolu-1}), (\ref{compute-supersolu-2}), (\ref{compute-supersolu-3}) and (\ref{compute-supersolu-4}) that
\begin{equation}\label{compute-supersolu-5}
N\overline{V}\geq -\partial_{\xi}U_{e(x)}Ch(\alpha x)-\alpha(C_5+C_6)h(\alpha x)-\varepsilon\alpha^2C_7h(\alpha x)+f(x,y,U_{e(x)})-f(x,y,\overline{V}).
\end{equation}
Since $U_e(-\infty,x,y)=1$ and $U_e(+\infty,x,y)=0$ for any $e\in\mathbb{S}^{N-1}$ and $(x,y)\in\mathbb{T}^N$, there is $R>0$ such that $0<U_e(\overline{\xi},x,y)\leq\sigma/2$ and $1-\sigma/2\leq U_e(\overline{\xi},x,y)<1$ for $\overline{\xi}\geq R$ and $\overline{\xi}\leq -R$ respectively, where $\sigma$ is defined in (A3). Take $\varepsilon_0\leq\sigma/2$, one has that $\overline{V}(t,x,y)<\sigma/2+\varepsilon\leq\sigma$ for $(t,x,y)\in\mathbb{R}\times\mathbb{R}^N$ such that $\overline{\xi}(t,x,y)\geq R$, since $\varepsilon\leq\varepsilon_0$ and $h(x)<1$. By (A3), we obtain
\begin{equation}\label{compute-supersolu-6}
f(x,y,U_{e(x)})-f(x,y,\overline{V})\geq\kappa\varepsilon h(\alpha x).
\end{equation} 
Thus, one can take $\alpha(\varepsilon)>0$ sufficiently small such that 
\begin{equation}\label{compute-supersolu-7}
-\alpha(C_5+C_6)-\varepsilon\alpha^2C_7+\kappa\varepsilon\geq 0,\text{ for all } 0<\alpha\leq\alpha(\varepsilon).
\end{equation}
Then, by (\ref{compute-supersolu-5}), (\ref{compute-supersolu-6}) and $\partial_{\xi}U_{e(x)}<0$, one has that
\begin{equation*}
N\overline{V}\geq \left(-\alpha(C_5+C_6)-\varepsilon\alpha^2C_7+\kappa\varepsilon\right)h(\alpha x)\geq 0.
\end{equation*}
Similarly, one can prove that $N\overline{V}\geq 0$ for $(t,x,y)\in\mathbb{R}\times\mathbb{R}^N$ such that $\overline{\xi}(t,x,y)\leq -R$. As for $(t,x,y)\in\mathbb{R}\times\mathbb{R}^N$ such that $-R\leq\overline{\xi}(t,x,y)\leq R$, it follows from Lemma \ref{pulsating-lemma-2} that there is $r>0$ such that $-\partial_{\xi}U_{e(x)}\geq r.$ Let $M:=\max_{(x,y,u)\in\mathbb{T}^N\times[0,1]}|f_u(x,y,u)|$, which can be reached since $f_u(x,y,u)$ is of class $C^2$ in $u\in\mathbb{R}$. Then one has that
\begin{equation}\label{compute-supersolu-8}
f(x,y,U_{e(x)})-f(x,y,\overline{V})\geq-\varepsilon M h(\alpha x).
\end{equation}
Now we take $\varepsilon_0=\text{min}\{rC_8/(\kappa+M),\sigma/2\}$ and $\alpha(\varepsilon_0)$ small enough such that (\ref{compute-supersolu-7}) holds, then it follows from (\ref{compute-supersolu-5}), (\ref{compute-supersolu-7}) and (\ref{compute-supersolu-8}) that
\begin{equation*}
	\begin{split}
N\overline{V}
& \geq \left(rC-\alpha(C_5+C_6)-\varepsilon\alpha^2C_7-\varepsilon M\right)h(\alpha x)\\
& \geq \left(rC-\varepsilon(\kappa+M)\right)h(\alpha x)\geq 0.
    \end{split}
\end{equation*}
for $0<\varepsilon\leq\varepsilon_0$ and $0<\alpha\leq\alpha(\varepsilon_0)$.

As a consequence, $N\overline{V}(t,x,y)\geq 0$ for all $(t,x,y)\in\mathbb{R}\times\mathbb{R}^N$, which means that  $\overline{V}$ is a supersolution of (\ref{initial-equation}).

Case~2: $\underline{W}$ is a subsolution. We now show that $\underline{W}(t,x,y)$ is a subsolution of (\ref{initial-equation}) by the same approach as above and only adjusting some signs. By the definition of $\overline{W}$, direct computation gives
	\begin{eqnarray}\label{compute-subsolu-1}
		\begin{array}{lll}
			N\underline{W}&:=&\underline{W}_t-\nabla_{x,y}\cdot(A(x,y)\nabla_{x,y}\underline{W})-f(x,y,\underline{W})\\
			&\leq& \partial_{\xi}U_{e(-x)}(\underline{\xi}_t+c_{e(-x)})
			-\sum_{i,j=1}^{N-1}a_{ij}U_{e(-x)}''\cdot \partial_{x_i}e(-x)\cdot\partial_{x_j}e(-x)\\
			&&+2\sum_{i,j=1}^{N-1}a_{ij}\partial_{x_j}\underline{\xi}\partial_{\xi}U_{e(-x)}'\cdot\partial_{x_i}e(-x) 
			+2\sum_{i,j=1}^{N-1}a_{ij}\partial_{x_j}U_{e(-x)}'\cdot \partial_{x_i}e(-x)\\
			&&+\sum_{i,j=1}^{N-1}a_{ij}U_{e(-x)}'\cdot \partial_{x_ix_j}e(-x) +\sum_{i=1}^{N-1}||A||_{C^1(\mathbb{R}^N)}U_{e(-x)}'\cdot\partial_{x_i}e(-x)\\
			&&+N^2||A||_{C^1(\mathbb{R}^N)}|\partial_{\xi}U_{e(-x)}(\nabla_{x,y}\underline{\xi}-e(-x))| +N^2||A||_{C^0(\mathbb{R}^N)}|\partial_{\xi}U_{e(-x)}\nabla^2_{x,y}\underline{\xi}|\\
			&&-\partial_{\xi\xi}U_{e(-x)}(\nabla_{x,y}\underline{\xi} A\nabla_{x,y}\underline{\xi}-e(-x)Ae(-x))-2\nabla_{x,y}\partial_{\xi}U_{e(-x)}A(\nabla_{x,y}\underline{\xi}-e(-x))\\
			&& +\varepsilon\nabla_{x,y}\cdot(A(\nabla_xh(-\alpha x),0))+f(x,y,U_{e(-x)})-f(x,y,\underline{W}),
		\end{array}
	\end{eqnarray}
where $\partial_{\xi}U_{e(-x)}$, $\partial_{\xi\xi}U_{e(-x)}$, $\nabla_{x,y}\partial_{\xi}U_{e(-x)}$, $
U_{e(-x)}''\cdot \partial_{x_i}e(-x)\cdot\partial_{x_j}e(-x)$, $
\partial_{\xi}U_{e(-x)}'\cdot\partial_{x_i}e(-x)$, $
\partial_{x_j}U_{e(-x)}'\cdot \partial_{x_i}e(-x)$, $ U_{e(-x)}'\cdot \partial_{x_ix_j}e(-x)$, $U_{e(-x)}'\cdot\partial_{x_i}e(-x)$ are taking values at $(\underline{\xi}(t,x,y),x,y)$, $A$ is taking values at $(x,y)$ and $\underline{W}$, $\underline{\xi}$, $\underline{\xi}_t$, $\nabla_{x,y}\underline{\xi}$, $\nabla^2_{x,y}\underline{\xi}$ are taking values at $(t,x,y)$.

It is easy to see that estimates (\ref{compute-supersolu-2}), (\ref{compute-supersolu-3}) and (\ref{compute-supersolu-4}) actually hold for all $(\underline{\xi},x,y)\in\mathbb{R}\times\mathbb{T}^N$. Then it follows from $\partial_{\xi}U_{e(-x)}<0$, (\ref{claim-sub}), (\ref{compute-supersolu-2}), (\ref{compute-supersolu-3}), (\ref{compute-supersolu-4}) and (\ref{compute-subsolu-1}) that
\begin{equation*}
	N\underline{W}\leq \partial_{\xi}U_{e(-x)}Ch(-\alpha x)+\alpha(C_5+C_6)h(-\alpha x)+\varepsilon\alpha^2C_7h(-\alpha x)+f(x,y,U_{e(-x)})-f(x,y,\underline{W}).
\end{equation*}
Dividing $\mathbb{R}\times\mathbb{R}^N$ into three parts, and applying (A3) as we do for Case~1, one eventually gets $N\underline{W}(t,x,y)\leq 0$ for all $(t,x,y)\in\mathbb{R}\times\mathbb{R}^N$.

This completes the proof.
\end{proof}

\vskip 0.3cm

\section{Existence of curved fronts}
In this section, we show the existence and uniqueness of curved fronts satisfying Theorem \ref{thm-existence} and Theorem~\ref{coro}. 

Let $e_0=(0,0,\cdots,1)$ and $\{e_i\}_{i=1}^n$ satisfy (i)-(iv) of Theorem~\ref{thm-existence}. Consider \eqref{initial-equation}. For $\varepsilon>0$ and $\alpha>0$, let $\overline{V}(t,x,y)$ and $\overline{\xi}(t,x,y)$ be defined by \eqref{eq:supersolution} and \eqref{uxilxi} respectively where $y=\varphi(x)$ is the surface determined by $\{e_i\}_{i=1}^n$ given as in Section~\ref{subs:surface}, $e(x)$ is given by \eqref{e(x)} and $\overline{c}=\hat{c}$ is given by (ii) of Theorem~\ref{thm-existence}. 

\begin{lemma}\label{lemma3.1}
There exists $C>0$ such that
$$-\overline{\xi}_t-c_{e(x)}\geq Ch(\alpha x),\quad \text{ for some }C>0.$$
\end{lemma}

\begin{proof}
By \eqref{surface}, we can easily get the derivatives of $\varphi(x)$. In particular,
\begin{equation*}
\nabla\varphi(x)=-\frac{\sum_{i=1}^{n}e^{-\hat{q}_i}\nu_i\cos\theta_i}{\sum_{i=1}^{n}e^{-\hat{q}_i}\sin\theta_i}.
\end{equation*}
 Then, since $e_i=(\nu_i\cos\theta_i,\sin\theta_i)$ and by \eqref{e(x)}, one can compute that
\begin{equation}\label{eq:e(x)}
e(x)=\frac{1}{\sqrt{1+|\nabla\varphi(\alpha x)|^2}}\left(\frac{\sum_{i=1}^{n}e^{-\hat{q}_i(\alpha x)}\nu_i\cos\theta_i}{\sum_{i=1}^{n}e^{-\hat{q}_i(\alpha x)}\sin\theta_i},\, 1\right)=\sum_{i=1}^n\tau_i(x)e_i,
\end{equation}
where
\begin{equation*}
\tau_i(x)=\frac{e^{-\hat{q}_i(\alpha x)}}{\sqrt{1+|\nabla\varphi(\alpha x)|^2} \sum_{j=1}^{n}e^{-\hat{q}_j(\alpha x)}\sin\theta_j}\geq 0, \text{ for every }i\in\{1,\cdots,n\}.
\end{equation*}
Write $\tau_i(x)$ by $\tau_i$ for short.
By the definition of $\mathcal{Q}$ and $\mathcal{L}(\mathcal{Q})$ in (iii) of Theorem~\ref{thm-existence}, one knows that $x\cdot e(x)=\sum_{i=1}^n\tau_i  x\cdot e_i\geq 0$ for all $x\in \mathcal{Q}$, which immediately implies $e(x)\in\mathcal{L}(\mathcal{Q})$.

Notice that the $e(x)\cdot e_0=1/\sqrt{1+|\nabla\varphi(\alpha x)|^2}$ and $g(x)=c_{\frac{x}{|x|}}/(\frac{x}{|x|}\cdot e_0)$. Thus,
\begin{equation}\label{claim-conclusion}
-\overline{\xi}_t-c_{e(x)}=\frac{\hat{c}}{\sqrt{1+|\nabla\varphi(\alpha x)|^2}}-c_{e(x)}=\frac{1}{\sqrt{1+|\nabla\varphi(\alpha x)|^2}}\Big(g(e_i)-g(e(x))\Big),
\end{equation}
for $x\in\widehat{Q}_i$. In the following proof we will focus on establishing estimates of $\tau_i$ and $|e(x)-e_i|$ in different cases.

Recall that for every $i\in\{1,\cdots,n\}$, $\hat{q}_i(x)\to0$ and $\hat{q}_j(x)\to+\infty$ for all $j\neq i$ and $x\in\widehat{Q}_i$ as ${\rm dist}(x,\widehat{\mathcal{R}})\to+\infty$. That is, $e(x)\rightarrow e_i$ for $x\in\widehat{Q}_i$ as ${\rm dist}(x,\widehat{\mathcal{R}})\to+\infty$. Then, it follows from \eqref{esti-surface-2} that there is $B>0$ such that
\begin{align*}
|\tau_i -1|=& \left|\frac{e^{-\hat{q}_i(\alpha x)}(1-\sqrt{1+|\nabla\varphi(\alpha x)|^2} \sin\theta_i)-\sqrt{1+|\nabla\varphi(\alpha x)|^2} \sum_{j\neq i}e^{-\hat{q}_j(\alpha x)}\sin\theta_j}{\sqrt{1+|\nabla\varphi(\alpha x)|^2} \sum_{j=1}^{n}e^{-\hat{q}_j(\alpha x)}\sin\theta_j} \right|\\
\le& B h(\alpha x),
\end{align*}
and there are $B_2>B_1>0$ such that
\begin{equation*}
B_1 h(\alpha x)\le \sum_{j\neq i}\tau_j\le B_2 h(\alpha x), 
\end{equation*}
for $x\in\widehat{Q}_i$ and ${\rm dist}(x,\widehat{\mathcal{R}})$ sufficiently large.
Obviously, there is $\widetilde{B}>0$ such that
$$|e(x)- e_i|\le \widetilde{B} h(\alpha x), $$ 
for $x\in\widehat{Q}_i$ and ${\rm dist}(x,\widehat{\mathcal{R}})$ sufficiently large. Using the Taylor expansion, one has
\begin{align*}
g(e_i)-g(e(x))=&-\nabla g(e_i)\cdot(e(x)-e_i)-o(|e(x)-e_i|)\\
=& -(\tau_i-1)\nabla g(e_i)\cdot e_i -\sum_{j\neq i} \tau_j \nabla g(e_i)\cdot e_j-o(h(\alpha x)),
\end{align*}
 for $x\in\widehat{Q}_i$ as ${\rm dist}(x,\widehat{\mathcal{R}})\to+\infty.$
By the definition of $g$ and (iv) of Theorem~\ref{thm-existence}, one can easily show that $\nabla g(e_i)\cdot e_i=0$ and there is $B'>0$ such that 
\begin{equation}\label{claim-case1-conclusion}
g(e_i)-g(e(x))\ge B' h(\alpha x),
\end{equation}
 for $x\in\widehat{Q}_i$ and ${\rm dist}(x,\widehat{\mathcal{R}})$ sufficiently large.

On the other hand, if $x\in\widehat{Q}_i$ and ${\rm dist}(x, \widehat{\mathcal{R}})$ is bounded, we only have to prove that there is $\delta>0$ such that $|e(x)-e_i|\geq\delta$. Because, if $|e(x)-e_i|\geq\delta$,  if follows from (iii) of Theorem~\ref{thm-existence} and the continuity of $g$ that there is $\delta_1>0$ such that
\begin{equation}\label{claim-case2-conclusion}
g(e_i)-g(e(x))\geq\delta_1h(\alpha x),\quad \text{ for }x\in\widehat{Q}_i\text{ and } {\rm dist}(x,\widehat{\mathcal{R}})<+\infty.
\end{equation}
Assume by contradiction that there is a sequence $\{x_n\}_{n\in\mathbb{N}}$ of $\mathbb{R}^{N-1}$ such that  ${\rm dist}(x_n,\widehat{\mathcal{R}})<+\infty$ and $|e(x_n)-e_i|\to 0$ as $n\to+\infty$. Two cases may occur: either
\begin{equation}\label{claim-case1}
|x_n|<+\infty\text{ for all large $n$},
\end{equation}
or
\begin{equation}\label{claim-case2}
|x_n|\to+\infty\text{ as }n\to+\infty. 
\end{equation}

If (\ref{claim-case1}) happens, there is $x_0$ such that $x_n$, up to extraction of a subsequence, converges to $x_0$. Since $e(x)$ is continuous with respect to $x$ and $e(x_n)\to e_i$ as $n\to+\infty$, one has that $e(x_0)=e_i$. The tangent plane of the surface $y=\varphi(\alpha x)/\alpha$  on  the point $(x_0,\varphi(\alpha x_0)/\alpha)$ can be written as
\begin{equation*}
T=\{(x,y)\in\mathbb{R}^N;\, x\cdot\nu_i\cos\theta_i+y\sin\theta_i-\hat{q}_i(\alpha x_0)/\alpha=0\}.
\end{equation*}
Since $(x_0,\varphi(\alpha x_0)/\alpha)$ is in $\mathcal{Q}$, there is $\delta_2>0$ such that $\hat{q}_i(\alpha x_0)\geq\delta_2$.  By Lemma \ref{surface-lemma}, one knows that the surface $y=\varphi(\alpha x)/\alpha$ is approaching the facets of $\partial \mathcal{Q}$ exponentially as $x$ diverging from $\widehat{\mathcal{R}}$. Therefore, there are points $(x_1,\varphi(\alpha x_1)/\alpha)$, $(x_2,\varphi(\alpha x_2)/\alpha)$ on the surface $y=\varphi(\alpha x)/\alpha$ such that $x_1\cdot\nu_i\cos\theta_i+\varphi(\alpha x_1)/\alpha\cdot \sin\theta_i$ is small and $x_2\cdot\nu_i\cos\theta_i+\varphi(\alpha x_2)/\alpha\cdot \sin\theta_i$ is large. It means that points $(x_1,\varphi(\alpha x_1)/\alpha)$, $(x_2,\varphi(\alpha x_2)/\alpha)$ are on the opposite sides of the tangent plane $T$, which contradicts the convexity of $\varphi(x)$.

We now deal with (\ref{claim-case2}). One can easily check that ${\rm dist}(x_n,\widehat{\mathcal{R}})<+\infty$ and $|x_n|\to+\infty$ implies that $\hat{q}_i(\alpha x_n)$, $\hat{q}_j(\alpha x_n)$ are bounded for some $j\neq i$, and $\hat{q}_k(\alpha x_n)\to 0$ for all $k\in\{1,\cdots,n\}\setminus\{i,j\}$, as $n\to+\infty$. By \eqref{eq:e(x)}, one has that $\tau_i(x_n)$, $\tau_j(x_n)$ are bounded away from $0$ for all $n$ and  $\tau_k(x_n)\rightarrow 0$ as $n\rightarrow +\infty$.  Then, it is not possible that $e(x_n)=\sum_{i=1}^n\tau_i(x_n)e_i\rightarrow e_i$.

Then, the conclusion of (\ref{claim-super}) follows from (\ref{claim-conclusion}), (\ref{claim-case1-conclusion}) and (\ref{claim-case2-conclusion}) immediately.
\end{proof}

\begin{lemma}\label{supV}
Assume that (i)-(iv) of Theorem~\ref{thm-existence} hold.
There exist $\varepsilon_0>0$ and $\alpha(\varepsilon_0)>0$ such that for $0<\varepsilon<\varepsilon_0$ and $0<\alpha<\alpha(\varepsilon_0)$, 
 the function $\overline{V}(t,x,y)$ is a supersolution of (\ref{initial-equation}) satisfying
 \begin{equation}\label{lemma-supersolu-1}
|\overline{V}(t,x,y)-\underline{V}(t,x,y)|\leq2\varepsilon, \quad\text{ uniformly as }d((t,x,y),\mathcal{R}+\hat{c}te_0)\to+\infty,
\end{equation}
and
\begin{equation}\label{lemma-supersolu-2}
\overline{V}(t,x,y)\geq\underline{V}(t,x,y),\quad \text{ for }(t,x,y)\in\mathbb{R}\times\mathbb{R}^N,
\end{equation}
where $\underline{V}(t,x,y)$ defined by \eqref{subsolution} is a subsolution of (\ref{initial-equation}).
\end{lemma}

\begin{proof}
By Lemmas~\ref{supersolution} and \ref{lemma3.1}, one immediately has that $\overline{V}(t,x,y)$ is a supersolution of \eqref{initial-equation} for $0<\varepsilon<\varepsilon_0$ and $0<\alpha<\alpha(\varepsilon_0)$.

\textbf{Step~1: proof of \eqref{lemma-supersolu-1}.}  For $(t,x,y)\in\mathbb{R}\times\mathbb{R}^N$ such that $d((t,x,y),\partial\mathcal{Q}+\hat{c}te_0)<+\infty$ and $d((t,x,y),\mathcal{R}+\hat{c}te_0)\geq\rho$ as $\rho\to+\infty$, there is some $i\in\{1,\cdots,n\}$ such that 
\begin{equation*}
x\cdot\nu_i\cos\theta_i+(y-\hat{c}t)\sin\theta_i\text{ is bounded and }x\cdot\nu_j\cos\theta_j+(y-\hat{c}t)\sin\theta_j\to+\infty\text{ for all }j\neq i.
\end{equation*}
Since $\hat{c}=c_{e_i}/e_i\cdot e_0=c_{e_i}/\sin\theta_i$ for all $i$, this implies that 
$$\underline{V}(t,x,y)=\max\limits_{1\leq k\leq n}\{U_{e_k}(x\cdot\nu_k\cos\theta_k+(y-\hat{c}t)\sin\theta_k,x,y)\}=U_{e_i}(x\cdot\nu_i\cos\theta_i+(y-\hat{c}t)\sin\theta_i,x,y),$$ 
since $U_e(+\infty,x,y)=0$ for any $e\in\mathbb{S}^{N-1}$ and $(x,y)\in\mathbb{T}^N$. It further means that $x\in\widehat{Q}_i$ and $d(x,\widehat{\mathcal{R}})\to+\infty$. Then by $\hat{q}_i(\alpha x)\to 0$ and (\ref{esti-surface-2}), one has that $\overline{\xi}(t,x,y)\to x\cdot\nu_i\cos\theta_i+(y-\hat{c}t)\sin\theta_i$. Thus,
\begin{equation}\label{step2-1}
|U_{e_i}(\overline{\xi}(t,x,y),x,y)-U_{e_i}(x\cdot\nu_i\cos\theta_i+(y-\hat{c}t)\sin\theta_i,x,y)|\leq\frac{\varepsilon}{2}.
\end{equation}
Again by (\ref{esti-surface-2}), one knows that $e(x)\to(\nu_i\cos\theta_i,\sin\theta_i)=e_i$. Then it follows from Lemma \ref{pulsating-lemma-3} and the boundedness of $U_{e_i}$ that,
\begin{equation}\label{step2-2}
|U_{e(x)}(\overline{\xi}(t,x,y),x,y)-U_{e_i}(\overline{\xi}(t,x,y),x,y)|\leq||U_{e_i}'||\cdot|e(x)-e_i|+o(|e(x)-e_i|)\leq\frac{\varepsilon}{2}.
\end{equation}
By the definition of $\overline{V}$, (\ref{step2-1}) and (\ref{step2-2}), one has that
\begin{equation*}
	\begin{split}
|\overline{V}(t,x,y)-\underline{V}(t,x,y)|
=&|\overline{V}(t,x,y)-U_{e_i}(x\cdot\nu_i\cos\theta_i+(y-\hat{c}t)\sin\theta_i,x,y)|\\
\leq& |U_{e(x)}(\overline{\xi}(t,x,y),x,y)-U_{e_i}(\overline{\xi}(t,x,y),x,y)|+|U_{e_i}(\overline{\xi}(t,x,y),x,y)\\
&-U_{e_i}(x\cdot\nu_i\cos\theta_i+(y-\hat{c}t)\sin\theta_i,x,y)|+\varepsilon |h(\alpha x)|\\
\leq& 2\varepsilon.
	\end{split}
\end{equation*}
For $(t,x,y)\in\mathbb{R}\times\mathbb{R}^N$ such that $d((t,x,y),\partial\mathcal{Q}+\hat{c}te_0)\geq\rho$ as $\rho\to+\infty$, one has that $$\min\limits_{1\leq i\leq n}\{x\cdot\nu_i\cos\theta_i+(y-\hat{c}t)\sin\theta_i\}\to+\infty\text{ or }-\infty.$$
We consider $\min\limits_{1\leq i\leq n}\{x\cdot\nu_i\cos\theta_i+(y-\hat{c}t)\sin\theta_i\}\to+\infty$ first. Then $0<\underline{V}(t,x,y)\leq \varepsilon/2$ since $U_e(+\infty,x,y)=0$ for any $e\in\mathbb{S}^{N-1}$ and $(x,y)\in\mathbb{T}^N$. We point out that the surface $y=\varphi(x)$ is bounded away from $\partial\mathcal{Q}$. Otherwise, there is a sequence $\{x_k\}_{k\in\mathbb{N}}$ such that 
$\min\limits_{1\leq i\leq n}\{\hat{q}_i(x_k)\}\to+\infty\text{ or }-\infty$ as $k\to+\infty$, which contradicts the function $\sum_{i=1}^ne^{-\hat{q}_i(x_k)}=1$. This property immediately implies that $\alpha(y-\hat{c}t)-\varphi(\alpha x)\to+\infty$, and then $\overline{\xi}(t,x,y)\to +\infty$. Therefore, $0<U_{e(x)}(\overline{\xi}(t,x,y),x,y)\leq\varepsilon/2$ and it follows that
\begin{equation*}
|\overline{V}(t,x,y)-\underline{V}(t,x,y)|\leq|U_{e(x)}(\overline{\xi}(t,x,y),x,y)-\underline{V}(t,x,y)|+\varepsilon|h(\alpha x)|\leq 2\varepsilon.
\end{equation*}

Similar arguments can be applied for $\min\limits_{1\leq i\leq n}\{x\cdot\nu_i\cos\theta_i+(y-\hat{c}t)\sin\theta_i\}\to-\infty$ to get the above inequality.

\textbf{Step~2: proof of (\ref{lemma-supersolu-2}). } It is sufficient to prove that for all $i\in\{1,\cdots,n\}$, the inequality 
$$\overline{V}(t,x,y)\geq U_{e_i}(x\cdot\nu_i\cos\theta_i+(y-\hat{c}t)\sin\theta_i,x,y)$$ holds for $(t,x,y)\in\mathbb{R}\times\mathbb{R}^N$. We now write $u_i(t,x,y)=U_{e_i}(x\cdot\nu_i\cos\theta_i+(y-\hat{c}t)\sin\theta_i,x,y)$ for short. Since $U_e(-\infty,x,y)=1$ and $U_e(+\infty,x,y)=0$ for any $e\in\mathbb{S}^{N-1}$ and $(x,y)\in\mathbb{T}^N$, One can check from Step~1 that $u_i(t,x,y)$ is an invasion of $0$ by $1$ with sets
\begin{equation*}
	\left\{
	\begin{aligned}
		& \Gamma_t=\{(t,x,y)\in\mathbb{R}\times\mathbb{R}^N; x\cdot\nu_i\cos\theta_i+(y-\hat{c}t)\sin\theta_i=0\},\\
		& \Omega_{t}^{\pm}=\{(t,x,y)\in\mathbb{R}\times\mathbb{R}^N; x\cdot\nu_i\cos\theta_i+(y-\hat{c}t)\sin\theta_i< 0,> 0\}.
	\end{aligned}
	\right.
\end{equation*}
and $\overline{V}(t,x,y)$ is a sup-invasion of $0$ by $1$ with sets
\begin{equation*}
	\left\{
	\begin{aligned}
		& \widetilde{\Gamma}_t=\{(t,x,y)\in\mathbb{R}\times\mathbb{R}^N; \min\limits_{1\leq i\leq n}\{x\cdot\nu_i\cos\theta_i+(y-\hat{c}t)\sin\theta_i\}=0\},\\
		& \widetilde{\Omega}_{t}^{\pm}=\{(t,x,y)\in\mathbb{R}\times\mathbb{R}^N; \min\limits_{1\leq i\leq n}\{x\cdot\nu_i\cos\theta_i+(y-\hat{c}t)\sin\theta_i\}< 0,> 0\}.
	\end{aligned}
	\right.
\end{equation*}
It is easy to see that $\widetilde{\Omega}_{t}^{-}\subset \Omega_{t}^{-}$, and then one knows from Lemma \ref{pulsating-lemma-5} that there exists $T_{min}$ defined as
\begin{equation*}
	T_{min}=\inf\{T\in\mathbb{R};\overline{V}(t+T,x,y)\geq u_i(t,x,y)\text{ for }(t,x,y)\in\mathbb{R}\times\mathbb{R}^N\}
\end{equation*}
such that $\overline{V}(t+T_{min},x,y)\geq u_i(t,x,y)$ in $\mathbb{R}\times\mathbb{R}^N$.
By (\ref{lemma-supersolu-1}), one has that $\overline{V}(t,x,y)\to U_{e_i}(x\cdot\nu_i\cos\theta_i+(y-\hat{c}t)\sin\theta_i,x,y)$ for $(t,x,y)\in\mathbb{R}\times\mathbb{R}^N$ such that $d((t,x,y),\widetilde{Q}_i+\hat{c}te_0)<+\infty$ and $d((t,x,y),\mathcal{R}+\hat{c}te_0)\to+\infty$. This means that $0\leq T_{min}<+\infty$ and we only need to prove $T_{min}=0$.

By the definition of sub-invasion, sup-invasion and (\ref{transition4}), for $0<\sigma<1/2$ given in (A3), there exist $m>0$ and $\widetilde{m}>0$ such that
\begin{equation*}
	\left\{
	\begin{aligned}
		&\forall t\in\mathbb{R}, \forall (x,y)\in\Omega_{t}^{+}, (d((x,y),\Gamma_t)\geq m)\Rightarrow(u_i(t,x,y)\geq 1-\sigma/2),\\
		&\forall t\in\mathbb{R}, \forall (x,y)\in\Omega_{t}^{-}, (d((x,y),\Gamma_t)\geq m)\Rightarrow(u_i(t,x,y)\leq \sigma),\\		
	\end{aligned}
	\right.
\end{equation*}
and
\begin{equation*}
	\left\{
	\begin{aligned}
		&\forall t\in\mathbb{R}, \forall (x,y)\in\widetilde{\Omega}_{t}^{+}, (d((x,y),\widetilde{\Gamma}_t)\geq \widetilde{m})\Rightarrow(\overline{V}(t,x,y)\geq 1-\sigma/2),\\
		&\forall t\in\mathbb{R}, \forall (x,y)\in\widetilde{\Omega}_{t}^{-}, (d((x,y),\widetilde{\Gamma}_t)\geq \widetilde{m})\Rightarrow(\overline{V}(t,x,y)\leq \sigma).\\		
	\end{aligned}
	\right.
\end{equation*}
Define $\omega_m^{\pm}$, $\Omega_{\widetilde{m}}^{\pm}$ as follows:
\begin{equation*}
	\begin{split}
		& \omega_m^{\pm}:=\{(t,x,y)\in\mathbb{R}\times\mathbb{R}^N; (x,y)\in\Omega_{t}^{\pm}, d((x,y),\Gamma_t)\geq m\},\\
		& \Omega_{\widetilde{m}}^{\pm}:=\{(t,x,y)\in\mathbb{R}\times\mathbb{R}^N; (x,y)\in\widetilde{\Omega}_{t}^{\pm}, d((x,y),\widetilde{\Gamma}_t)\geq \widetilde{m}\}.
	\end{split}
\end{equation*}

Assume now by contradiction that $T_{min}>0$, there are two cases: either
\begin{equation}\label{step3-5}
\inf\{\overline{V}(t+T_{min},x,y)-u_i(t,x,y); (t,x,y)\in\mathbb{R}\times\mathbb{R}^N\setminus(\omega_{m}^-\cup\Omega_{\widetilde{m}}^+)\}>0,
\end{equation}
or
\begin{equation}\label{step3-6}
\inf\{\overline{V}(t+T_{min},x,y)-u_i(t,x,y); (t,x,y)\in\mathbb{R}\times\mathbb{R}^N\setminus(\omega_{m}^-\cup\Omega_{\widetilde{m}}^+)\}=0.
\end{equation}
If (\ref{step3-5}) happens, there is $\eta_0>0$ such that for any $\eta\in(0,\eta_0]$, it holds that $\overline{V}(t+T_{min}-\eta,x,y)-u_i(t,x,y)\geq 0$ for all $(t,x,y)\in\mathbb{R}\times\mathbb{R}^N\setminus(\omega_{m}^-\cup\Omega_{\widetilde{m}}^+)\}$. Especially, one has that $\overline{V}(t+T_{min}-\eta,x,y)-u_i(t,x,y)\geq 0$ on $\partial\omega_{m}^-\cup\partial\Omega_{\widetilde{m}}^+$. Notice that for any $\eta\in(0,\eta_0]$, $u_i(t,x,y)\leq\sigma$ in $\omega_m^-$ and $\overline{V}(t+T_{min}-\eta,x,y)\geq 1-\sigma$ in $\Omega_{\widetilde{m}}^+$ (even if it means decreasing $\eta_0>0$). Then by the same arguments as Steps~2-3 in the proof of Lemma \ref{pulsating-lemma-5} , one can get that $\overline{V}(t+T_{min}-\eta,x,y)\geq u_i(t,x,y)$ for $(t,x,y)\in\mathbb{R}\times\mathbb{R}^N$, which contradicts the definition of $T_{min}$. 

If (\ref{step3-6}) happens, there exists a sequence $\{(t_n,x_n,y_n)\}_{n\in\mathbb{N}}$ of $\mathbb{R}\times\mathbb{R}^N\setminus(\omega_{m}^-\cup\Omega_{\widetilde{m}}^+)$ such that
\begin{equation*}
\overline{V}(t_n+T_{min},x_n,y_n)-u_i(t_n,x_n,y_n)\to 0\text{ as }n\to+\infty.
\end{equation*}
This immediately implies that $d((t_n,x_n,y_n),\mathcal{R}+\hat{c}te_0)<+\infty$, since that for $T_{min}>0$,
\begin{equation*}
	\left\{
	\begin{aligned}
		 \overline{V}(t+T_{min},x,y)\to u_i(t+T_{min},x,y)>u_i(t,x,y)\text{ for all }(t,x,y)\in\mathbb{R}\times\mathbb{R}^N\text{ such that }\\ d((t,x,y),\widetilde{Q}_i+\hat{c}te_0)<+\infty\text{ and }d((t,x,y),\mathcal{R}+\hat{c}te_0)\to+\infty,\\
		 \overline{V}(t+T_{min},x,y)\to u_j(t+T_{min},x,y)>u_i(t,x,y)\text{ for all }(t,x,y)\in\mathbb{R}\times\mathbb{R}^N\text{ such that }\\ d((t,x,y),\widetilde{Q}_j+\hat{c}te_0)<+\infty\text{ and }d((t,x,y),\mathcal{R}+\hat{c}te_0)\to+\infty, \forall j\in\{1,\cdots,n\}\setminus\{i\}, \\	
		 \overline{V}(t+T_{min},x,y)\to \underline{V}(t+T_{min},x,y)>u_i(t,x,y)\text{ for all }(t,x,y)\in\mathbb{R}\times\mathbb{R}^N\text{ such that }\\ d((t,x,y),\partial\mathcal{Q}+\hat{c}te_0)\to+\infty.\\	
	\end{aligned}
	\right.
\end{equation*}

 Then one can take another sequence $\{(t_n-1,x_n',y_n')\}$ such that $d((t_n-1,x_n',y_n'),(t_n,x_n,y_n))$ is bounded and $d((t_n-1,x_n',y_n'),\mathcal{R}+\hat{c}te_0)$ large enough such that $\overline{V}(t_n-1+T_{min},x_n',y_n')>u_i(t_n-1,x_n',y_n')$. However, it follows from linear parabolic estimates that $\overline{V}(t_n-1+T_{min},x_n',y_n')-u_i(t_n-1,x_n',y_n')\to 0$, which is a contradiction. Therefore, $T_{min}=0$ and $\overline{V}(t,x,y)\geq U_{e_i}(x\cdot\nu_i\cos\theta_i+(y-\hat{c}t)\sin\theta_i,x,y)$ for all $(t,x,y)\in\mathbb{R}\times\mathbb{R}^N$.

By similar arguments, we can prove that 
$$\overline{V}(t,x,y)\geq U_{e_j}(x\cdot\nu_j\cos\theta_j+(y-\hat{c}t)\sin\theta_j,x,y),$$
 for any other $j\in\{1,\cdots,n\}$. In conclusion, $\overline{V}(t,x,y)\geq \underline{V}(t,x,y)$ for all $(t,x,y)\in\mathbb{R}\times\mathbb{R}^N$.
This completes the proof.
\end{proof}

\vskip 0.3cm

\begin{proof}[Proof of Theorem \ref{thm-existence}]
Once we have the supersolution $\overline{V}(t,x,y)$ at hand, we can easily prove the existence for $e_0=(0,0,\cdots,1)$. Then, for general $e_0\in\mathbb{S}^{N-1}$, we can change variables of $\overline{V}(t,x,y)$ for the supersolution.
 
{ \bf Step~1: the existence for $e_0=(0,0,\cdots,1)$.}
 Let $u_n(t,x,y)$ be the solution of (\ref{initial-equation}) for $t\geq -n$ with initial data
\begin{equation*}
	u_n(-n,x,y)=\underline{V}(-n,x,y),\quad (x,y)\in\mathbb{R}^N.
\end{equation*}
Then, it follows from Lemma \ref{supV} and the comparison principle that,
\begin{equation*}
	\underline{V}(t,x,y)\leq u_n(t,x,y)\leq\overline{V}(t,x,y),\text{ for }t\geq -n\text{  and }(x,y)\in\mathbb{R}^N.
\end{equation*}
Notice that the sequence $u_n(t,x,y)$ is increasing in $n$ since $\underline{V}(t,x,y)$ is a subsolution of \eqref{initial-equation}. Letting $n\to+\infty$ and by parabolic estimates, the sequence converges to an entire solution $V(t,x,y)$ of (\ref{initial-equation}) satisfying
\begin{equation*}
	\underline{V}(t,x,y)\leq V(t,x,y)\leq\overline{V}(t,x,y),\text{ for }(t,x,y)\in\mathbb{R}\times\mathbb{R}^N.
\end{equation*}
Then by \eqref{lemma-supersolu-1} and arbitrariness of $\varepsilon>0$, one gets that \eqref{thm-existence-1} holds. By \eqref{thm-existence-1} and the definition of $\underline{V}(t,x,y)$, one can easily know that $V(t,x,y)$ is a transition front with sets 
\begin{equation*}
	\left\{
	\begin{aligned}
		& \Gamma_t=\{(x,y)\in\mathbb{R}^N;\, \min\limits_{1\leq i\leq n}\{x\cdot\nu_i\cos\theta_i+(y-\hat{c}t)\sin\theta_i\}=0\}=\partial\mathcal{Q}+\hat{c} t e_0,\\
		& \Omega_{t}^{-}=\{(x,y)\in\mathbb{R}^N;\, \min\limits_{1\leq i\leq n}\{x\cdot\nu_i\cos\theta_i+(y-\hat{c}t)\sin\theta_i\}> 0\}=\mathcal{Q}+\hat{c} t e_0,\\
		& \Omega_{t}^{+}=\{(x,y)\in\mathbb{R}^N;\, \min\limits_{1\leq i\leq n}\{x\cdot\nu_i\cos\theta_i+(y-\hat{c}t)\sin\theta_i\}< 0\}=\R^N\setminus \overline{\mathcal{Q}}  +\hat{c} t e_0.
	\end{aligned}
	\right.
\end{equation*}
 Obviously, $V(t,x,y)$ is an invasion of $0$ by $1$.  By \cite[Theorem~1.11]{MR2898886}, one gets that $V_t(t,x,y)>0$ for all $(t,x,y)\in\mathbb{R}\times\mathbb{R}^N$. 

{ \bf Step~2: the existence for general $e_0\in \mathbb{S}^{N-1}$.} In this case, the subsolution $\underline{V}(t,x)$ of \eqref{1.1} is still given by \eqref{subsolution}. Let $Y=(x\cdot e_0) e_0$ and $X=x-(x\cdot e_0) e_0$. The supersolution $\overline{V}(t,x)$ of \eqref{1.1} is now given by 
\begin{equation*}
	\overline{V}(t,x):=\min\{U_{e(x)}(\overline{\xi}(t,x),x)+\varepsilon h(\alpha X),1\},
\end{equation*}
where 
$$\overline{\xi}(t,x)=\frac{Y-\hat{c}t-\varphi(\alpha X)/\alpha}{\sqrt{1+|\nabla\varphi(\alpha X)|^2}}.$$
No additional difficulties arise in the checking process of the supersolution $\overline{V}(t,x)$. Then, as Step~1, we can show the existence and monotonicity of the curved front.

{ \bf Step~3: the uniqueness of $V(t,x)$.}
Suppose that $V(t,x)$ and $\widetilde{V}(t,x)$ are both curved fronts of  \eqref{1.1} satisfying (\ref{thm-existence-1}). Then, they are all invasions of $0$ by $1$ with sets \eqref{TF-sets}. It follows from (\ref{transition4}) that there is $R>0$ such that $0<V(t,x), \widetilde{V}(t,x)\leq\sigma$ for $(t,x)\in\omega^-$ and $1-\sigma\leq V(t,x), \widetilde{V}(t,x)<1$ for $(t,x)\in\omega^+$, where
\begin{equation*}
	\begin{split}
		& \omega^{\pm}:=\{(t,x)\in\mathbb{R}\times\mathbb{R}^N; \min\limits_{1\leq i\leq n}\{x\cdot e_i-c_{e_i} t\}<-R,>R\},\\
		& \omega:=\{(t,x)\in\mathbb{R}\times\mathbb{R}^N; -R\leq \min\limits_{1\leq i\leq n}\{x\cdot e_i -c_{e_i} t\}\leq R\}.
	\end{split}
\end{equation*}
Since $\widetilde{\Omega}^-_{t+t_0}\subset\Omega_t^-$ by taking some $t_0>0$, one knows from Lemma \ref{pulsating-lemma-5} that there exists $$T_{min}=\inf\{T\in\mathbb{R};\, \widetilde{V}(t+T,x)\geq V(t,x)\text{ for }(t,x)\in\mathbb{R}\times\mathbb{R}^N\}$$
such that $\widetilde{V}(t+T_{min},x)\geq V(t,x)$ in $\mathbb{R}\times\mathbb{R}^N$. Since $V$, $\widetilde{V}$ satisfy (\ref{thm-existence-1}) and $\underline{V}(t+T,x)>\underline{V}(t,x)$ for $T>0$, we get that $T_{min}$ is well-definde and $0\leq T_{min}<+\infty$. Then one can follow the same arguments as Step~2 in the proof of Lemma \ref{supV} by considering that 
$$\inf\limits_{\omega}\{\widetilde{V}(t+T_{min},x)-V(t,x)\}>0 \hbox{ or }\inf\limits_{\omega}\{\widetilde{V}(t+T_{min},x)-V(t,x)\}=0$$ 
and eventually get $T_{min}=0$, which means $\widetilde{V}(t,x)\geq V(t,x)$ for all $(t,x)\in\mathbb{R}\times\mathbb{R}^N$. One can ues same arguments and change positions of $\widetilde{V}(t,x)$ and $V(t,x)$ to get that $\widetilde{V}(t,x)\leq V(t,x)$ for all $(t,x)\in\mathbb{R}\times\mathbb{R}^N$. Thus, $\widetilde{V}(t,x)\equiv V(t,x)$.

 This completes the proof.
\end{proof}
\vskip 0.3cm

Now, let $e_0=(0,0,\cdots,1)$ and $\{e_i\}_{i=1}^n$ satisfy (i)-(iv) of Theorem~\ref{coro}. Consider \eqref{initial-equation}. For $\varepsilon>0$ and $\alpha>0$, let $\underline{W}(t,x,y)$ and $\underline{\xi}(t,x,y)$ be defined by \eqref{eq:subsolution} and \eqref{uxilxi} respectively where $y=\varphi(x)$ is the surface determined by $\{e_i\}_{i=1}^n$ given as in Section~\ref{subs:surface} and $\overline{c}=\hat{c}$ is given by (ii) of Theorem~\ref{coro}. 

\begin{lemma}\label{3.3}
There exists $C>0$ such that
$$\underline{\xi}_t+c_{e(-x)}\geq Ch(-\alpha x),\quad \text{ for some }C>0.$$
\end{lemma}

\begin{proof}
By the definition of $\underline{\xi}(t,x,y)$, we have 
$$\underline{\xi}_t+c_{e(-x)}=-\frac{\hat{c}}{\sqrt{1+|\nabla\varphi(-\alpha x)|^2}}+c_{e(-x)}=\frac{1}{\sqrt{1+|\nabla\varphi(-\alpha x)|^2}}(g(e(-x))-g(e_i)),$$
for $-x\in \widehat{Q}_i$. By noticing that (iii) and (iv) of Theorem~\ref{coro} are reversed inequalities of  (iii) and (iv) of Theorem~\ref{thm-existence}, one can follow the same proof of Lemma~\ref{lemma3.1} to show that $g(e(-x))-g(e_i)\ge Ch(-\alpha x)$ for some $C>0$.
\end{proof}

\begin{lemma}\label{3.4}
Assume that (i)-(iv) of Theorem~\ref{coro} hold.
There exist $\varepsilon_0>0$ and $\alpha(\varepsilon_0)>0$ such that for $0<\varepsilon<\varepsilon_0$ and $0<\alpha<\alpha(\varepsilon_0)$, 
 the function $\underline{W}(t,x,y)$ is a subsolution of (\ref{initial-equation}) satisfying
\begin{equation*}
	|\underline{W}(t,x,y)-\overline{W}(t,x,y)|\leq2\varepsilon, \quad\text{ uniformly as }d((t,x,y),-\mathcal{R}+\hat{c} t e_0)\to+\infty,
\end{equation*}
and
\begin{equation*}
	\underline{W}(t,x,y)\leq\overline{W}(t,x,y),\quad\text{ for }(t,x,y)\in\mathbb{R}\times\mathbb{R}^N,
\end{equation*}
where $\overline{W}(t,x,y)=\min\limits_{1\leq i\leq n}\{U_{e_i}((x,y)\cdot e_i -c_{e_i} t,x,y)\}$ is a supersolution of (\ref{initial-equation}).
\end{lemma}

\begin{proof}
By Lemma \ref{supersolution} and Lemma \ref{3.3}, one immediately knows $\underline{W}(t,x,y)$ is a subsolution of (\ref{initial-equation}). By noticing that $e(-x)\to e_i$ and $\underline{\xi}(t,x,y)\to x\cdot \nu_i \cos\theta_i+(y-\hat{c}t) \sin\theta_i$ if $-x\in\hat{Q}_i$ and $d(-x,\hat{\mathcal{R}})\to+\infty$, one can follow Step~1 of the proof of Lemma \ref{supV} to get that $	|\underline{W}(t,x,y)-\overline{W}(t,x,y)|\leq2\varepsilon$ uniformly for $d((t,x,y),-\mathcal{R}+\hat{c} t e_0)$ sufficiently large. Due to these two properties and $\hat{c}=c_{e_i}/e_i\cdot e_0=c_{e_i}/\sin\theta_i$, we observe that for each $i\in\{1,\cdots,n\}$, $u_i(t,x,y):=U_{e_i}(x\cdot\nu_i\cos\theta_i+(y-\hat{c}t)\sin\theta_i,x,y)$ is an invasion of $0$ by $1$ with sets
\begin{equation*}
	\left\{
	\begin{aligned}
		& \Gamma_t=\{(t,x,y)\in\mathbb{R}\times\mathbb{R}^N; x\cdot\nu_i\cos\theta_i+(y-\hat{c}t)\sin\theta_i=0\},\\
		& \Omega_{t}^{\pm}=\{(t,x,y)\in\mathbb{R}\times\mathbb{R}^N; x\cdot\nu_i\cos\theta_i+(y-\hat{c}t)\sin\theta_i< 0,> 0\}.
	\end{aligned}
	\right.
\end{equation*}
and $\underline{W}(t,x,y)$ is a sub-invasion of $0$ by $1$ with sets
\begin{equation*}
	\left\{
	\begin{aligned}
		& \widetilde{\Gamma}_t=\{(t,x,y)\in\mathbb{R}\times\mathbb{R}^N; \max\limits_{1\leq i\leq n}\{x\cdot\nu_i\cos\theta_i+(y-\hat{c}t)\sin\theta_i\}=0\},\\
		& \widetilde{\Omega}_{t}^{\pm}=\{(t,x,y)\in\mathbb{R}\times\mathbb{R}^N; \max\limits_{1\leq i\leq n}\{x\cdot\nu_i\cos\theta_i+(y-\hat{c}t)\sin\theta_i\}< 0,> 0\}.
	\end{aligned}
	\right.
\end{equation*}
It is easy to see that $\Omega_{t}^{-}\subset \widetilde{\Omega}_{t}^{-}$, and therefore one knows from Lemma \ref{pulsating-lemma-5} that there exists $T_{min}$ defined as
\begin{equation*}
	T_{min}=\inf\{T\in\mathbb{R}; u_i(t+T,x,y)\geq \underline{W}(t,x,y)\text{ for }(t,x,y)\in\mathbb{R}\times\mathbb{R}^N\}
\end{equation*}
such that $u_i(t+T_{min},x,y)\geq \underline{W}(t,x,y)$ in $\mathbb{R}\times\mathbb{R}^N$. Furthermore, there exists a sequence $\{(t_n,x_n,y_n)\}_{n\in\mathbb{N}}$ such that
$$\{d(x_n,\widetilde{\Gamma}_{t_n})\}\text{ is bounded and }u_i(t_n+T_{min},x_n,y_n)-\underline{W}(t_n,x_n,y_n)\to 0\text{ as }n\to+\infty.$$
Then one can follow Step~2 of the proof of Lemma \ref{supV} by using linear parabolic estimates to get a contradiction if $T_{min}>0$. Thus, $\underline{W}(t,x,y)\leq U_{e_i}(x\cdot\nu_i\cos\theta_i+(y-\hat{c}t)\sin\theta_i,x,y)$ for $(t,x,y)\in\mathbb{R}\times\mathbb{R}^N$ and for every $i\in\{1,\cdots,n\}$. This leads to the conclusion.
\end{proof}

\vskip 0.3cm

\begin{proof}[Proof of Theorem~\ref{coro}]
Let $u_n(t,x,y)$ be the solution of (\ref{initial-equation}) for $t\geq -n$ with initial data
\begin{equation*}
	u_n(-n,x,y)=\overline{W}(-n,x,y),\quad (x,y)\in\mathbb{R}^N.
\end{equation*}
Then, it follows from Lemma \ref{3.4} and the comparison principle that,
\begin{equation*}
	\underline{W}(t,x,y)\leq u_n(t,x,y)\leq\overline{W}(t,x,y),\text{ for }t\geq -n\text{  and }(x,y)\in\mathbb{R}^N.
\end{equation*}
Notice that the sequence $u_n(t,x,y)$ is decreasing in $n$ since $\overline{W}(t,x,y)$ is a supersolution of \eqref{initial-equation}. Letting $n\to+\infty$ and by parabolic estimates, the sequence converges to an entire solution $W(t,x,y)$ of (\ref{initial-equation}) satisfying
\begin{equation*}
	\underline{W}(t,x,y)\leq W(t,x,y)\leq\overline{W}(t,x,y),\text{ for }(t,x,y)\in\mathbb{R}\times\mathbb{R}^N.
\end{equation*}
Again by Lemma \ref{3.4} and arbitrariness of $\varepsilon>0$, one gets that \eqref{coro-1} holds. By \eqref{coro-1} and the definition of $\overline{W}(t,x,y)$, one can easily check that $W(t,x,y)$ is a transition front with sets 
\begin{equation*}
	\left\{
	\begin{aligned}
		& \Gamma_t=\{(x,y)\in\mathbb{R}^N;\, \max\limits_{1\leq i\leq n}\{x\cdot\nu_i\cos\theta_i+(y-\hat{c}t)\sin\theta_i\}=0\}=-\partial\mathcal{Q}+\hat{c} t e_0,\\
		& \Omega_{t}^{-}=\{(x,y)\in\mathbb{R}^N;\, \max\limits_{1\leq i\leq n}\{x\cdot\nu_i\cos\theta_i+(y-\hat{c}t)\sin\theta_i\}> 0\}=-\R^N\setminus \overline{\mathcal{Q}}  +\hat{c} t e_0,\\
		& \Omega_{t}^{+}=\{(x,y)\in\mathbb{R}^N;\, \max\limits_{1\leq i\leq n}\{x\cdot\nu_i\cos\theta_i+(y-\hat{c}t)\sin\theta_i\}< 0\}=-\mathcal{Q}+\hat{c} t e_0.
	\end{aligned}
	\right.
\end{equation*}
Obviously, $W(t,x,y)$ is an invasion of $0$ by $1$.  By \cite[Theorem~1.11]{MR2898886}, one gets that $W_t(t,x,y)>0$ for all $(t,x,y)\in\mathbb{R}\times\mathbb{R}^N$. If both $W$ and $\widetilde{W}$ are curved fronts of \eqref{1.1} satisfying \eqref{coro-1}, one can trivially replace $V$ and $\widetilde{V}$ with $W$ and $\widetilde{W}$ in Step~2 of the proof of Theorem \ref{thm-existence} to get $W\equiv\widetilde{W}$.

For general $e_0\in\mathbb{S}^{N-1}$, the supersolution of \eqref{1.1} is given by 
$$\overline{W}(t,x)=\min\limits_{1\leq i\leq n}\{U_{e_i}(x\cdot e_i -c_{e_i} t,x)\},$$
and the subsolution of \eqref{1.1} is given by 
$$\underline{W}(t,x)=\max\{U_{e(-x)}(\underline{\xi}(t,x),x)-\varepsilon h(-\alpha X),0\},$$
where $Y=(x\cdot e_0) e_0$, $X=x-(x\cdot e_0) e_0$ and
$$\underline{\xi}(t,x)=\frac{Y-\hat{c}t+\varphi(-\alpha X)/\alpha}{\sqrt{1+|\nabla\varphi(-\alpha X)|^2}}.$$
With the same arguments, we can show the existence and uniqueness of curved fronts for \eqref{1.1}.
\end{proof}


\section{Stability of curved fronts}
In this section, we prove the asymptotic stability of curved fronts in Theorem \ref{thm-existence} and \ref{coro}. We still consider $e_0=(0,0,\cdots,1)$ and \eqref{initial-equation} for convenience.

%

Assume that conditions of Theorem \ref{thm-existence} hold. In order to prove the stability of $V(t,x,y)$, we first construct subsolutions which can be compared with $U_{e_i}((x,y)\cdot e_i -c_{e_i} t,x,y)$ for each $i\in\{1,\cdots,n\}$.  Fix any $i\in \{1,\cdots, n\}$ and consider the facet $\widetilde{Q}_i$ of the polytope $\mathcal{Q}$. Take $\lambda_i>0$  such that  $(1-\lambda_i)e_i\cdot e_0 -\lambda_i  e_j\cdot e_0=(1-\lambda_i)\sin\theta_i-\lambda_{i}\sin\theta_j>0$ for all $j\neq i$. For each $j\in\{1,\cdots,n\}\setminus \{i\}$, define
$$e_{ij}:=\frac{(1-\lambda_i)e_i-\lambda_{i}e_j}{|(1-\lambda_{i})e_i-\lambda_{i}e_j|}.$$
Let $e_{ii}=e_i$. Then, $e_{ij}$ is defined for every $j\in\{1,\cdots,n\}$.
Obviously, $e_{ij}\cdot e_0>0$ for all $j\in \{1,\cdots,n\}$. Let $Q_{ij}$  be the hyperplane determined by $e_{ij}$. In particular, $Q_{ii}=Q_i$. 
Then let $\mathcal{Q}_i$ be  the polytope  enclosed by $Q_{ij}$ and let $\widetilde{Q}_{ij}=\partial\mathcal{Q}_i\cap Q_{ij}$ be the facets of $\mathcal{Q}_i$. Let $H_{jk}=\widetilde{Q}_{ij}\cap\widetilde{Q}_{ik}$ be its ridges for $j\neq k$ and $\mathcal{H}$ be the set of all ridges. In fact, the intersection of $Q_{ij}$ and $Q_i$ is the same as that of $Q_j$ and $Q_i$. Let $\widehat{Q}_{ij}$ and $\widehat{H}_{ij}$ be the projection of $\widetilde{Q}_{ij}$ and $H_{ij}$ on $x$-plane respectively. 
\begin{figure}[H]
	\centering
	\includegraphics[scale=0.3]{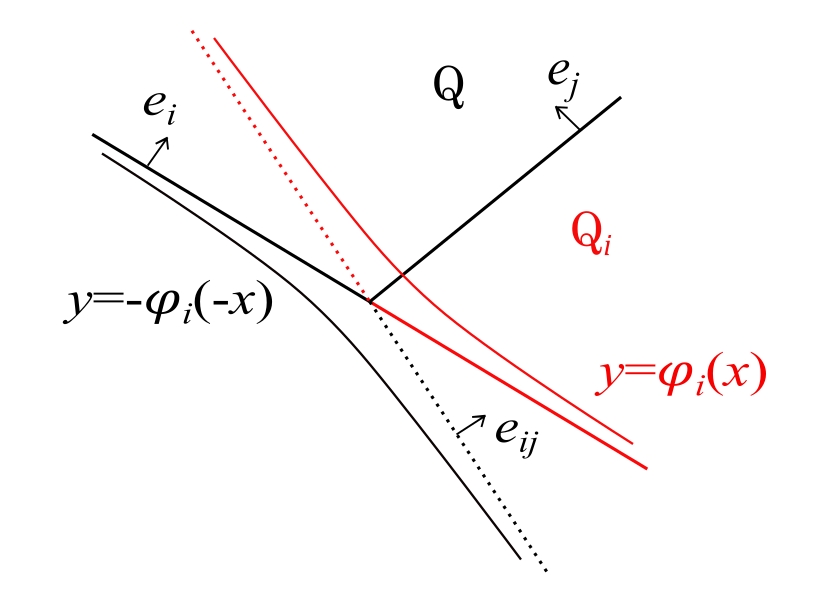}
	\caption{an example of $\mathcal{Q}_i$ and $y=\varphi_i(x)$ in dimension 2.}
\end{figure}
Let $q_{ij}(x,y):=(x,y)\cdot e_{ij}$ and $y=\varphi_i(x)$ be the surface determined by $\sum_{j=1}^n e^{-q_{ij}(x,y)}=1$ as in Section~\ref{subs:surface}. One can see Figure~2 for an example of $\mathcal{Q}_i$ and $y=\varphi_i(x)$ in dimension 2. We write $q_{ij}(x,\varphi_i(x))$ by $\hat{q}_{ij}(x)$. Let $e^i(x)$ be given by \eqref{e(x)} by replacing $\varphi$ with $\varphi_i$. Define
\begin{equation*}
	\underline{V}_i(t,x,y)=U_{e^i(-x)}(\underline{\xi}_i(t,x,y),x,y)-\varepsilon \hat{h}(-\alpha x),
\end{equation*}
where
\begin{equation*}
	\underline{\xi}_i(t,x,y)=\frac{y-\hat{c}t+\varphi_i(-\alpha x)/\alpha}{\sqrt{1+|\nabla\varphi_i(-\alpha x)|^2}}\text{ and }\hat{h}(x)=\sum_{k,l\in\{1,\cdots,n\},k\neq l}e^{-(\hat{q}_{ik}(x)+\hat{q}_{il}(x))}.
\end{equation*}
We prove that $\underline{V}_i(t,x,y)$ is a subsolution of (\ref{initial-equation}) for small enough $\lambda_i$, $\varepsilon$ and $\alpha$.

\begin{lemma}\label{lemma4.1}
For $\lambda_i$ small enough, it holds 
$$
(\underline{\xi}_i)_t+c_{e^i(-x)}\ge C \hat{h}(-\alpha x),\quad \text{ for some }C>0.$$
\end{lemma}

\begin{proof}
Consider every fixed $i\in\{1,\cdots,n\}$. Let $e_{ij}=(\widetilde{\nu}_j\cos\widetilde{\theta}_j,\sin\widetilde{\theta}_j)$, that is,
$$\widetilde{\nu}_j\cos\widetilde{\theta}_j=\frac{(1-\lambda_i)\nu_i\cos\theta_i-\lambda_i\nu_j\cos\theta_j}{|(1-\lambda_{i})e_i-\lambda_{i}e_j|} \hbox{ and } \sin\widetilde{\theta}_j=\frac{(1-\lambda_i)\sin\theta_i-\lambda_i\sin\theta_j}{|(1-\lambda_{i})e_i-\lambda_{i}e_j|},$$
for $j\neq i$. 
After tedious calculation, one can show that $|e_{ij}-e_i|=O(\lambda_i)$ for every $j\neq i$ and $|\nabla \varphi_i +\nu_i\cot\theta_i|=O(\lambda_i)$ as $\lambda_i\rightarrow 0$.
As \eqref{eq:e(x)}, we compute that 

\begin{equation*}
	e^i(-x)=\sum_{j=1}^n\tau_{ij}(-x)e_{ij}, 
	\end{equation*}
	where
	\begin{equation*}
	\tau_{ij}(-x)=\frac{e^{-\hat{q}_{ij}(-\alpha x)}}{\sqrt{1+|\nabla\varphi_i(-\alpha x)|^2} \sum_{j=1}^{n}e^{-\hat{q}_{ij}(-\alpha x)}\sin\widetilde{\theta}_j}
	\geq 0.
\end{equation*}

On the other hand,
$$(\underline{\xi}_i)_t+c_{e^i(-x)}=\frac{1}{\sqrt{1+|\nabla\varphi_i(-\alpha x)|^2}}\Big(g(e^i(-x))-g(e_i)\Big),\quad \hbox{for $x\in\R^{N-1}$}.$$
By the proof of Lemma~\ref{lemma3.1}, we know that $e^i(-x)\rightarrow e_{ij}$ for $-x\in  \widehat{Q}_{ij}$ as ${\rm dist}(x,\widehat{\mathcal{H}})\rightarrow +\infty$. More subtle calculation gives that, for $-x\in  \widehat{Q}_{ij}$ and ${\rm dist}(x,\widehat{\mathcal{H}})$ sufficiently large, there are positive constants $B_1$, $B_2$, $B_3$ such that
$$ \quad B_1 \hat{h}(-\alpha x)\le \sum_{j\neq i} \tau_{ij}(-x)\le B_2 \hat{h}(-\alpha x),$$
and
$$ \left|\sum_{j=1}^n \frac{\tau_{ij}(-x)}{|(1-\lambda_{i})e_i-\lambda_{i}e_j|}-1\right|\le B_3\lambda_i\hat{h}(-\alpha x).$$
Thus, there exists $\widetilde{B}>0$ such that  
\begin{align*}
|e^i(-x)-e_i|=&\left|\Big(\sum_{j=1}^n \frac{\tau_{ij}(-x)}{|(1-\lambda_{i})e_i-\lambda_{i}e_j|}-1\Big) e_i -\frac{\lambda_{i}}{|(1-\lambda_{i})e_i-\lambda_{i}e_j|} \sum_{j\neq i} \tau_{ij}(-x)  (e_i+e_j)\right|\\
\le& \widetilde{B}\lambda_i \hat{h}(-\alpha x),
\end{align*}
for $-x\in  \widehat{Q}_{ij}$ and ${\rm dist}(x,\widehat{\mathcal{H}})$ sufficiently large. Using the Taylor expansion, one has
\begin{equation*}
	\begin{aligned}
	g(e^i(-x))-g(e_i)&=\nabla g(e_i)\cdot(e^i(-x)-e_i)+o(|e^i(-x)-e_i|)\\
		&=\Big(\sum_{j=1}^n \frac{\tau_{ij}(-x)}{|(1-\lambda_{i})e_i-\lambda_{i}e_j|}-1\Big)\nabla g(e_i)\cdot e_i\\
		&\quad-\frac{\lambda_{i}}{|(1-\lambda_{i})e_i-\lambda_{i}e_j|} \sum_{j\neq i} \tau_{ij}\nabla g(e_i)\cdot (e_i+e_j)+o(|e^i(-x)-e_i|),
	\end{aligned}
\end{equation*}
for $-x\in  \widehat{Q}_{ij}$ as ${\rm dist}(x,\widehat{\mathcal{H}})\rightarrow +\infty$ or $\lambda_i\rightarrow 0$. Then it follows from $\nabla g(e_i)\cdot e_i=0$ and (iv) of Theorem \ref{thm-existence} that there is $B'>0$ such that
$$g(e^i(-x))-g(e_i)\ge B'\lambda_i \hat{h}(-\alpha x),$$
for $-x\in  \widehat{Q}_{ij}$ as ${\rm dist}(x,\widehat{\mathcal{H}})\rightarrow+\infty$ or $\lambda_i\rightarrow 0$.

If $-x\in  \widehat{Q}_{ij}$ and ${\rm dist}(x,\widehat{\mathcal{H}})<+\infty$, by replacing $\tau_i$ ($i=1,\cdots,n$) in the proof of Lemma~\ref{lemma3.1} with $\tau_{ij} $ ( $j=1,\cdots,n$), one has that
$$ \sum_{j\neq i} \tau_{ij}(x)\ge \delta>0,$$
for some $\delta>0$.
By taking $\lambda_i$ small enough and the Taylor expansion, one still has $B''>0$ such that
$$g(e^i(-x))-g(e_i)\ge B''\lambda_i \hat{h}(-\alpha x).$$
This completes the proof.
\end{proof}

\vskip 0.3cm

Take $\lambda_i>0$ such that Lemma~\ref{lemma4.1} holds.

\begin{lemma}\label{stab-subsolution}
	There exists $\varepsilon_0>0$ and $\alpha(\varepsilon_0)>0$ such that for $0<\varepsilon<\varepsilon_0$ and $0<\alpha<\alpha(\varepsilon_0)$, the function $\underline{V}_i(t,x,y)$ is a subsolution of (\ref{initial-equation}) satisfying
	\begin{equation*}
		\begin{split}
		|\underline{V}_i(t,x,y)-U_{e_i}(x\cdot\nu_i\cos\theta_i+y\sin\theta_i -c_{e_i}t,x,y)|\leq2\varepsilon,
		\end{split}
	\end{equation*}
	 uniformly for $(t,x,y)\in\R\times\left(-\widetilde{Q}_{ii}+\hat{c}te_0\right)$ as $d((t,x,y),-\mathcal{H}+\hat{c}te_0)\to+\infty$
	and
	\begin{equation*}
		\underline{V}_i(t,x,y)\leq U_{e_i}(x\cdot\nu_i\cos\theta_i+y\sin\theta_i -c_{e_i}t,x,y),\text{ for }(t,x,y)\in\mathbb{R}\times\mathbb{R}^N.
	\end{equation*}
\end{lemma}

\begin{proof}
For every $i\in\{1,\cdots,n\}$, we observe that $\underline{V}_i(t,x,y)$ is exactly identical in the structure with $\underline{W}(t,x,y)$ given by \eqref{eq:subsolution}. By Lemma~\ref{supersolution},  $\underline{V}_i(t,x,y)$ is a subsolution of \eqref{initial-equation}. Since $e_{ii}=e_i$, one can follow the same proof of Lemma \ref{3.4} to get that
\begin{equation*}
		\begin{split}
		|\underline{V}_i(t,x,y)-U_{e_i}(x\cdot\nu_i\cos\theta_i+y\sin\theta_i -c_{e_i}t,x,y)|\leq2\varepsilon,
		\end{split}
	\end{equation*}
	 uniformly for $(t,x,y)\in\R\times\left(-\widetilde{Q}_{ii}+\hat{c}te_0\right)$ as $d((t,x,y),-\mathcal{H}+\hat{c}te_0)\to+\infty$
	and
	\begin{equation*}
		\underline{V}_i(t,x,y)\leq U_{e_i}(x\cdot\nu_i\cos\theta_i+(y-\hat{c}t)\sin\theta_i,x,y),\text{ for }(t,x,y)\in\mathbb{R}\times\mathbb{R}^N,
	\end{equation*}
	by comparing $\underline{V}_i(t,x,y)$ and $ U_{e_i}(x\cdot\nu_i\cos\theta_i+(y-\hat{c}t)\sin\theta_i,x,y)$ using Lemma~\ref{pulsating-lemma-5}.
\end{proof}

\vskip 0.3cm

We also need the following properties by referring  to Lemma 2.7, Lemma 2.12 of \cite{preprint1} for the proofs.

\begin{lemma}\label{stab-lemma1}
For any $R>0$, it holds that
\begin{equation*}
	y-\frac{1}{\alpha}\varphi(\alpha x)\to-\infty\text{ for }(x,y)\in\mathbb{R}^N \text{ such that }d((x,y),\mathcal{R})<R\text{ as }\alpha\to 0^+,
\end{equation*}
where $\varphi(x)$ is the surface determined by $\{e_i\}_{i=1}^n$.
\end{lemma}

\begin{lemma}\label{stab-lemma2}
	For any $R>0$, it holds that
	\begin{equation*}
		y-\frac{1}{\alpha}\varphi_i(\alpha x)\to -\infty\text{ for }(x,y)\in\mathbb{R}^N \text{ such that }d((x,y),\mathcal{\mathcal{H}})<R\text{ as }\alpha\to 0^+,
	\end{equation*}
	where $\varphi_i(x)$ is the surface determined by $\{e_{ij}\}_{j=1}^n$.
\end{lemma}

By trivially extending the proof of Lemma 3.3 of \cite{MR4334733} to N-dimensional spaces, one has the following sub and supersolutions for the Cauchy problem (\ref{Chaucy-problem}).

\begin{lemma}\label{stab-lemma3}
If $u(t,x,y)\in C^{1,2}(\mathbb{R}\times\mathbb{R}^N)$ is a supersolution (resp. subsolution) of (\ref{initial-equation}) with $u_t>0$, and for any $\sigma_1\in(0,1/2)$, there is $k>0$ such that
\begin{equation*}
u_t\geq k,\quad \text{ for }(t,x,y)\in\mathbb{R}\times\mathbb{R}^N\text{ such that }\sigma_1\leq u(t,x,y)\leq 1-\sigma_1,
\end{equation*}
then for any $\delta\in(0,2/\sigma)$, where $\sigma$ is defined in (A3), there exist positive constants $\omega$ and $\lambda$ such that
\begin{equation*}
u^+(t,x,y):=u(t+\omega\delta(1-e^{-\lambda t}),x,y)+\delta e^{-\lambda t}\text{ is a supersolution of (\ref{initial-equation})}.
\end{equation*}
\begin{equation*}
	(\text{resp. }u^-(t,x,y):=u(t-\omega\delta(1-e^{-\lambda t}),x,y)-\delta e^{-\lambda t}\text{ is a subsolution of (\ref{initial-equation})}).
\end{equation*}
\end{lemma}

Now we are ready to prove the stability of the curved front.

\begin{proof}[Proof of Theorem \ref{thm-stability}]
For any $\delta\in(0,\sigma/2)$, take $\varepsilon_0$ and $\alpha(\varepsilon_0)$ such that Lemma \ref{supersolution} and Lemma
\ref{stab-subsolution} hold for $0<\varepsilon\leq\varepsilon_0$ and $0<\alpha\leq\alpha(\varepsilon_0)$. Let $\overline{V}(t,x,y)$ and $\underline{V}_i(t,x,y)$ be defined in Lemma \ref{supersolution} and Lemma
\ref{stab-subsolution}. Recall that $\underline{V}(t,x,y)=\max\limits_{1\leq i\leq n}\{U_{e_i}(x\cdot\nu_i\cos\theta_i+(y-\hat{c}t)\sin\theta_i,x,y)\}$. Let $\tilde{V}(t,x,y)=\max\limits_{1\leq i\leq n}\{\underline{V}_i(t,x,y)\}$. Then, it follows from Lemma \ref{stab-subsolution} and the definition of $\underline{V}(t,x,y)$ that
\begin{equation}\label{stab-proof-1}
\tilde{V}(t,x,y)\leq\underline{V}(t,x,y)\text{ for }(t,x,y)\in\mathbb{R}\times\mathbb{R}^N,
\end{equation}
and
\begin{equation}\label{eq:tildeV}
|\tilde{V}(t,x,y)-\underline{V}(t,x,y)|\leq 2\varepsilon,\quad \text{ uniformly as }d((t,x,y),\mathcal{R}+\hat{c} t e_0) \to+\infty.
\end{equation}

Let $u_0(x,y)$ be the initial value given in Theorem \ref{thm-stability}. One knows from (\ref{Chaucy-problem-initialvalue}) that, there is $R_{\delta}>0$ such that
\begin{equation}\label{stab-proof-2}
\underline{V}(0,x,y)-\frac{\delta}{2}\leq u_0(x,y)\leq \underline{V}(0,x,y)+\frac{\delta}{2}\text{ for }(x,y)\in\mathbb{R}^N\text{ such that }d((x,y),\mathcal{R})> R_{\delta}.
\end{equation}
In order to construct sub and supersolutions of (\ref{Chaucy-problem}), we define
\begin{equation*}
V^+=\overline{V}(t+\omega\delta(1-e^{-\lambda t}),x,y)+\delta e^{-\lambda t},
\end{equation*}
and
\begin{equation*}
V^-=\max\limits_{1\leq i\leq n}\{\underline{V}_i(t-\omega\delta(1-e^{-\lambda t}),x,y)-\delta e^{-\lambda t}\}.
\end{equation*}
Then, for $(x,y)\in\mathbb{R}^N$ such that $d((x,y),\mathcal{R})> R_{\delta}$, it follows from (\ref{lemma-supersolu-2}), (\ref{stab-proof-1}) and (\ref{stab-proof-2}) that, even if it means increasing $R_{\delta}$,
\begin{equation*}
V^+(0,x,y)\geq \underline{V}(0,x,y)+\delta\geq u_0(x,y)\geq\tilde{V}(0,x,y)-\delta=V^-(0,x,y).
\end{equation*}
As for $(x,y)\in\mathbb{R}^N$ such that $d((x,y),\mathcal{R})\leq R_{\delta}$, one knows from Lemma \ref{stab-lemma1} that $\overline{V}(0,x,y)\geq 1-\delta$ by taking $\alpha\in(0,\alpha(\varepsilon_0)]$ small enough. Thus, $V^+(0,x,y)=\overline{V}(0,x,y)+\delta\geq 1\geq u_0(x,y)$ since $0\leq u_0(x,y)\leq 1$. Similarly, Lemma \ref{stab-lemma2} implies that $\underline{V}_i(0,x,y)\leq\delta$ for each $i\in\{1,\cdots,n\}$. Therefore, $V^-(0,x,y)=\max\limits_{1\leq i\leq n}\{\underline{V}_i(0,x,y)-\delta\}\leq 0\leq u_0(x,y)$. In conclusion,
\begin{equation*}
V^-(0,x,y)\leq u_0(x,y)\leq V^+(0,x,y),  \text{ for }(x,y)\in\mathbb{R}^N.
\end{equation*}
By Lemma \ref{supersolution} and Lemma \ref{stab-subsolution}, one can easily check that the conditions given in Lemma \ref{stab-lemma3} are all hold for $\overline{V}(t,x,y)$ and $\underline{V}_i(t,x,y),i\in\{1,\cdots,n\}$. Thus, it follows from Lemma \ref{stab-lemma3} and comparison principle that
\begin{equation}\label{stab-proof-3}
V^-(t,x,y)\leq u(t,x,y)\leq V^+(t,x,y),\text{ for }t\geq0\text{ and }(x,y)\in\mathbb{R}^N.
\end{equation}
Take a sequence $\{t_n\}_{n\in\mathbb{N}}=\{nL_n/\hat{c}\}_{n\in\mathbb{N}}$ where $L_n$ is the period of $y$. Then, by parabolic estimates, the sequence $u_n(t,x,y):=u(t+t_n,x,y+nL_n)$ converges to a solution $u_{\infty}(t,x,y)$ of (\ref{initial-equation}), locally uniformly in $\mathbb{R}\times\mathbb{R}^N$. A direct computation gives that $\overline{V}(t+t_n,x,y+nL_n)=\overline{V}(t,x,y)$ and $\underline{V}_i(t+t_n,x,y+nL_n)=\underline{V}_i(t,x,y)$ for every $i\in\{1,\cdots,n\}$. By (\ref{stab-proof-3}), one has that
\begin{equation}\label{stab-proof-4}
\begin{aligned}
\max\limits_{1\leq i\leq n}\{\underline{V}_i(t-\omega\delta(1-e^{-\lambda(t+t_n)}),x,y)&-\delta e^{-\lambda(t+t_n)}\}\leq u_n(t,x,y)\\
&\leq \overline{V}(t+\omega\delta(1-e^{-\lambda(t+t_n)}),x,y)+\delta e^{-\lambda(t+t_n)}.
\end{aligned}
\end{equation}
Let $n\to+\infty$, then $t_n\to+\infty$ and $u_{\infty}(t,x,y)$ satisfies
\begin{equation}\label{stab-proof-5}
\tilde{V}(t-\omega\delta,x,y)\leq u_{\infty}(t,x,y)\leq \overline{V}(t+\omega\delta,x,y),\text{ for }(t,x,y)\in\mathbb{R}\times\mathbb{R}^N.
\end{equation}

Let $u(t,x,y;u_0(x,y))$ denote the solution of the following Cauchy problem for $t\geq t_0$
\begin{equation}\
	\left\{
	\begin{aligned}
		&u_t-\nabla_x\cdot(A(x,y)\nabla_{x,y} u)=f(x,y,u), &&t>t_0,(x,y)\in\mathbb{R}^N,\\
		&u(t_0,x,y)=u_0(x,y),  && t=t_0.
	\end{aligned}
	\right.
\end{equation}
Then, it follows from \eqref{lemma-supersolu-2} and comparison principle that
\begin{equation*}
\underline{V}(t+\omega\delta,x,y)\leq u(t,x,y;\overline{V}(t_0+\omega\delta,x,y))\leq \overline{V}(t+\omega\delta,x,y),
\end{equation*}
for $t\geq t_0$ and $(x,y)\in\mathbb{R}^N$. Combined with \eqref{lemma-supersolu-1} and \eqref{eq:tildeV} and by the uniqueness of the curved front, one has that
\begin{equation*}
|u(t,x,y;\overline{V}(t_0+\omega\delta,x,y))- V(t+\omega\delta,x,y)|\rightarrow 0 \text{ as $t\rightarrow +\infty$ for  }(x,y)\in\mathbb{R}^N.
\end{equation*}
Similarly,
\begin{equation*}
|u(t,x,y;\tilde{V}(t_0-\omega\delta,x,y))- V(t-\omega\delta,x,y)|\rightarrow 0 \text{ as $t\rightarrow +\infty$ for  }(x,y)\in\mathbb{R}^N.
\end{equation*}
By (\ref{stab-proof-5}) and comparison principle, one has that for $t\geq t_0$ and $(x,y)\in\mathbb{R}^N$
\begin{equation*}
u(t,x,y;\tilde{V}(t_0-\omega\delta,x,y))\leq u_{\infty}(t,x,y)\leq u(t,x,y;\overline{V}(t_0+\omega\delta,x,y)).
\end{equation*}
By passing to the limit $t_0\to-\infty$, one has that
\begin{equation*}
V(t-\omega\delta,x,y)\leq u_{\infty}(t,x,y)\leq V(t+\omega\delta,x,y)\text{ for }(t,x,y)\in\mathbb{R}\times\mathbb{R}^N.
\end{equation*}
Since $\delta$ can be taken arbitrary small, we then have $u_{\infty}(t,x,y)\equiv V(t,x,y)$. Thus, for any $\gamma>0$, it follows from (\ref{stab-proof-4}) and taking $\delta$ small enough that, there is $t'>0$ sufficiently large such that
\begin{equation*}
|u(t',x,y)-V(t',x,y)|\leq \gamma, \text{ for }(x,y)\in\mathbb{R}^N.
\end{equation*} 
By Lemma~\ref{stab-lemma3} and the comparison principle, one can show 
$$V(t'+t-\omega \gamma(1-e^{-\lambda t}),x,y)-\gamma e^{-\lambda t}\le u(t'+t,x,y)\le V(t'+t-\omega \gamma(1-e^{-\lambda t}),x,y)+\gamma e^{-\lambda t},$$
for $t\ge 0$ and $(x,y)\in\R^N$. 
By the arbitrariness of $\gamma$, one finally has that
\begin{equation*}
\lim\limits_{t\to+\infty}||u(t,x,y)-V(t,x,y)||_{L^{\infty}(\mathbb{R}^N)}=0.
\end{equation*}

For general $e_0\in\mathbb{S}^{N-1}$, one can modify the sub- and supersolutions in the same manner as in the proofs of Theorems~\ref{thm-existence} and \ref{coro}.

This completes the proof.
\end{proof}

\vskip 0.3cm

The idea for proving the stability of $W(t,x,y)$ is basically the same as for $V(t,x,y)$. Assume that conditions of Theorem \ref{coro} hold. We need to construct supersolutions which can be compared with $U_{e_i}((x,y)\cdot e_i -c_{e_i} t,x,y)$ for each $i\in\{1,\cdots,n\}$.  Define
\begin{equation*}
	\overline{W}_i(t,x,y)=U_{e^i(x)}(\overline{\xi}_i(t,x,y),x,y)-\varepsilon \hat{h}(\alpha x),
\end{equation*}
where
\begin{equation*}
	\overline{\xi}_i(t,x,y)=\frac{y-\hat{c}t-\varphi_i(\alpha x)/\alpha}{\sqrt{1+|\nabla\varphi_i(\alpha x)|^2}}\text{ and }\hat{h}(x)=\sum_{k,l\in\{1,\cdots,n\},k\neq l}e^{-(\hat{q}_{ik}(x)+\hat{q}_{il}(x))}.
\end{equation*}
As Lemma \ref{stab-subsolution},  we can show that $\overline{W}_i(t,x,y)$ is a supersolution of (\ref{initial-equation}) for small enough $\lambda_i$, $\varepsilon$ and $\alpha$.

\begin{lemma}
	There exists $\varepsilon_0>0$ and $\alpha(\varepsilon_0)>0$ such that for $0<\varepsilon<\varepsilon_0$ and $0<\alpha<\alpha(\varepsilon_0)$, the function $\overline{W}_i(t,x,y)$ is a supersolution of (\ref{initial-equation}) satisfying
	\begin{equation*}
		\begin{split}
		|\overline{W}_i(t,x,y)-U_{e_i}(x\cdot\nu_i\cos\theta_i+y\sin\theta_i -c_{e_i}t,x,y)|\leq2\varepsilon,
		\end{split}
	\end{equation*}
	 uniformly for $(t,x,y)\in\R\times\left(\widetilde{Q}_{ii}+\hat{c}te_0\right)$ as $d((t,x,y),\mathcal{H}+\hat{c}te_0)\to+\infty$
	and
	\begin{equation*}
		\overline{W}_i(t,x,y)\ge U_{e_i}(x\cdot\nu_i\cos\theta_i+y\sin\theta_i -c_{e_i}t,x,y),\text{ for }(t,x,y)\in\mathbb{R}\times\mathbb{R}^N.
	\end{equation*}
\end{lemma}

\vskip 0.3cm

\begin{proof}[Proof of Theorem \ref{coro-stability}] 
Since we have the subsolution $\underline{W}(t,x,y)$ and supersolutions $\overline{W}_i(t,x,y)$ at hand, the proof is then the same as for Theorem~\ref{thm-stability}. So we omit it. 
\end{proof}

\section{Appendix}

In this section, we will complete the proofs of Lemma \ref{pulsating-lemma-3} and Lemma \ref{pulsating-lemma-5}.
\vskip 0.3cm

\begin{proof}[Proof of Lemma \ref{pulsating-lemma-3}]
The main idea of this proof comes from \cite{MR3689331,MR3772873,MR4334733}, where the speed $c_e$ and the pulsating front profile $U_e$ for the case $A(x)\equiv I$ are proved to be doubly continuously Fr\'{e}chet differentiable with respect to $e\in\mathbb{S}^{N-1}$, under the normalization:  $\int_{\mathbb{R}^+\times\mathbb{T}^N}U_e^2(\xi,y)dyd\xi=1$ for all $e\in\mathbb{S}^{N-1}$. Our work is to generalize the Laplacian  to the general elliptic operator $\nabla\cdot(A(x)\nabla u)$. The main difference lies on the computation related to $\nabla \cdot(A(x)\nabla  u)$.

We first point out that the continuity of $(U_e,c_e)$ can be obtained directly from \cite[Theorem 1.3]{MR3772873} since the proof of it is essentially independent of the form of diffusion term. In order to get the differentiability of $(U_e,c_e)$, let us introduce some notations.  Define new variables $\xi=x\cdot e-c_et$ and $y=x$, then $\nabla_x U_e(x\cdot e-c_et,x)=\widetilde{\nabla}_eU(\xi,y)$, where $\widetilde{\nabla}_e:=e\cdot\partial_{\xi}+\nabla_y$. Let $L^2(\mathbb{R}\times\mathbb{T}^N)$, $H^1(\mathbb{R}\times\mathbb{T}^N)$ and $H^2(\mathbb{R}\times\mathbb{T}^N)$ be Banach spaces defined by
\begin{equation*}
\left\{
\begin{aligned}
	L^2(\mathbb{R}\times\mathbb{T}^N)&=\{u\in L^2_{loc}(\mathbb{R}\times\mathbb{T}^N); u(\xi,y+kL)=u(\xi,y)\text{ a.e.in }\mathbb{R}\times\mathbb{R}^N\\
	&\text{ for any }k\in\mathbb{Z}^N, L\in\mathbb{T}^N, \text{ and }u\in L^2(\mathbb{R}\times\Omega)\text{ for any bounded set }\Omega\subset\mathbb{R}^N\},\\	
		H^1(\mathbb{R}\times\mathbb{T}^N)&=\{u\in H^1_{loc}(\mathbb{R}\times\mathbb{T}^N); u(\xi,y+kL)=u(\xi,y)\text{ a.e.in }\mathbb{R}\times\mathbb{R}^N\\
	&\text{ for any }k\in\mathbb{Z}^N, L\in\mathbb{T}^N, \text{ and }u\in H^1(\mathbb{R}\times\Omega)\text{ for any bounded set }\Omega\subset\mathbb{R}^N\},\\	
		H^2(\mathbb{R}\times\mathbb{T}^N)&=\{u\in H^2_{loc}(\mathbb{R}\times\mathbb{T}^N); u(\xi,y+kL)=u(\xi,y)\text{ a.e.in }\mathbb{R}\times\mathbb{R}^N\\
	&\text{ for any }k\in\mathbb{Z}^N, L\in\mathbb{T}^N, \text{ and }u\in H^2(\mathbb{R}\times\Omega)\text{ for any bounded set }\Omega\subset\mathbb{R}^N\},
\end{aligned}
\right.
\end{equation*}
endowed with the norms
\begin{equation*}
	\begin{aligned}
		||u||_{L^2(\mathbb{R}\times\mathbb{T}^N)}=&\left(\int_{\mathbb{R}}\int_{\mathbb{T}^N}|u|^2dyd\xi\right)^{\frac{1}{2}},\\
		 ||u||_{H^1(\mathbb{R}\times\mathbb{T}^N)}=&||u||_{L^2(\mathbb{R}\times\mathbb{T}^N)}+||\partial_{\xi} u||_{L^2(\mathbb{R}\times\mathbb{T}^N)}+\sum_{i=1}^{N}||\partial_{ y_i}u||_{L^2(\mathbb{R}\times\mathbb{T}^N)},\\
		||u||_{H^2(\mathbb{R}\times\mathbb{T}^N)}=&||u||_{H^1(\mathbb{R}\times\mathbb{T}^N)}+||\partial_{\xi\xi}u||_{L^2(\mathbb{R}\times\mathbb{T}^N)}+\sum_{i=1}^{N}||\partial_{\xi}\partial_{y_i}u||_{L^2(\mathbb{R}\times\mathbb{T}^N)}\\
		&+\sum_{i,j=1}^{N}||\partial_{y_i}\partial_{y_j}u||_{L^2(\mathbb{R}\times\mathbb{T}^N)}.
	\end{aligned}
\end{equation*}
Define 
$$D:=\{v\in H^1(\mathbb{R}\times\mathbb{T}^N); \widetilde{\nabla}_e\cdot(A(y)\widetilde{\nabla}_ev)\in L^2(\mathbb{R}\times\mathbb{T}^N)\},$$
 endowed with the norm $||v||_{D}=||v||_{H^1(\mathbb{R}\times\mathbb{T}^N)}+||\widetilde{\nabla}_e\cdot(A(y)\widetilde{\nabla}_ev)||_{L^2(\mathbb{R}\times\mathbb{T}^N)}$.
Fix a real $\beta>0$ and for any $e\in\mathbb{S}^{N-1}$, define
\begin{equation*}
	M_{c,e}(v):=\widetilde{\nabla}_e\cdot(A(y)\widetilde{\nabla}_ev)+c\partial_{\xi} v-\beta v,\quad v\in D.
\end{equation*}

Now we sketch the proof of the differentiability of $(U_e,c_e)$ as follows.

\textit{ Step1: $M_{c,e}:D\to L^2(\mathbb{R}\times\mathbb{T}^N)$ is invertible.}

By recalling the periodicity of functions in $L^2(\mathbb{R}\times\mathbb{T}^N)$ and the positive definiteness of $A$, one can integrate $M_{c,e}(v)=0$ against $v$ in $L^2(\mathbb{R}\times\mathbb{T}^N)$ to obtain that $\int_{\mathbb{R}\times\mathbb{T}^N}(\widetilde{\nabla}_evA\widetilde{\nabla}_ev+\beta v^2)=0$, whence $v=0$. Applying the same arguments to the adjoint operator $M_{c,e}^*$ of $M_{c,e}$ gives that $\hbox{ker}(M_{c,e}^*)=\{0\}$. Assume now $M_{c,e}(v_n)=g_n$ with $v_n\in D$ and $g_n\to g$ in $L^2(\mathbb{R}\times\mathbb{T}^N)$. As in \cite[Lemma 3.1]{MR3689331}, testing $M_{c,e}(v_n)=g_n$ with $v_n$ and the symmetric difference quotient in $\xi$-direction, that is, $D_hv_n(\xi,y)=(v_n(\xi+h,y)-v_n(\xi-h,y))/(2h)$ where $h>0$, one can deduce that $v_n$ is bounded in $H^1(\mathbb{R}\times\mathbb{T}^N)$ for any $n\in\mathbb{N}$. Then by the Eberlein-Smulian Theorem, there is a subsequence $\{v_{n_i}\}_{i\in\mathbb{N}}$ such that $v_{n_i}\to v$ in $H^1(\mathbb{R}\times\mathbb{T}^N)$ weakly as $i\to +\infty$, which further means that $v\in D, M_{c,e}(v)=g$. Therefore, $M_{c,e}$ is invertible and there is a constant $C(\beta, c, \lambda_2)>0$ such that
\begin{equation}\label{appendix1-1}
	||M_{c,e}^{-1}(g)||_{H^1(\mathbb{R}\times\mathbb{T}^N)}\leq C||g||_{L^2(\mathbb{R}\times\mathbb{T}^N)}.
\end{equation}

\textit{Step2: for every $\{e_n\}\in\mathbb{S}^{N-1}$, $\{g_n\}\in L^2(\mathbb{R}\times\mathbb{T}^N)$, $\{c_n\}\in (0,+\infty)$ such that $e_n\to e$, $c_n\to c$, $||g_n-g||_{L^2(\mathbb{R}\times\mathbb{T}^N)}\to 0$ as $n\to +\infty$, there holds $M_{c_n,e_n}^{-1}(g_n)\to M_{c,e}^{-1}(g)$ in $H^1(\mathbb{R}\times\mathbb{T}^N)$ as $n\to +\infty.$} 

For any $\omega\in L^2(\mathbb{R}\times\mathbb{T}^N)$, let $p_n=M_{c_n,e_n}^{-1}(\omega), q_n=M_{c,e_n}^{-1}(\omega)$, then $M_{c,e_n}(p_n-q_n)=(c-c_n)\partial_{\xi}p_n$. Testing it with $(p_n-q_n)$ in $H^1(\mathbb{R}\times\mathbb{T}^N)$ and the same arguments in Step~1 gives that there is $C(\beta, c_n, c, \lambda_2)>0$ such that $||p_n-q_n||_{H^1(\mathbb{R}\times\mathbb{T}^N)}\leq C||\partial_{\xi}p_n||_{L^2(\mathbb{R}\times\mathbb{T}^N)}$. Now we consider $M_{c_n,e_n}(p_n)=\omega$, as in Step~1, one can test it with the symmetric difference quotient in $\xi$-direction $D_hp_n(\xi,y)=(p_n(\xi+h,y)-p_n(\xi-h,y))/(2h)$ to get that $||\partial_{\xi}p_n||_{L^2(\mathbb{R}\times\mathbb{T}^N)}\leq (1/c_n)||\omega||_{L^2(\mathbb{R}\times\mathbb{T}^N)}$. It follows that 
\begin{equation}\label{appendix1-2}
	||M_{c_n,e_n}^{-1}(\omega)-M_{c,e_n}^{-1}(\omega)||_{H^1(\mathbb{R}\times\mathbb{T}^N)}\leq C||\omega||_{L^2(\mathbb{R}\times\mathbb{T}^N)}.
\end{equation}
On the other hand, for any given $g\in L^2(\mathbb{R}\times\mathbb{T}^N)$, let $\omega_n=M_{c,e_n}^{-1}(g)$ and $\omega=M_{c,e}^{-1}(g)$. Then one gets from similar estimates in Step~1 that $\omega_n$ converge to $\omega$ strongly in $L^2_{loc}(\mathbb{R}\times\mathbb{T}^N)$ and weakly in $H^1(\mathbb{R}\times\mathbb{T}^N)$ as $n\to +\infty$. Note that $M_{c,e_n}(\omega_n-\omega)=-\widetilde{\nabla}_{e_n}\cdot(A\widetilde{\nabla}_{e_n}w)+\widetilde{\nabla}_{e}\cdot(A\widetilde{\nabla}_{e}w)$. Testing it with $\omega_n-\omega$ gives that
\begin{equation*}
	\int_{\mathbb{R}\times\mathbb{T}^N}\widetilde{\nabla}_{e_n}(\omega_n-\omega)A\widetilde{\nabla}_{e_n}(\omega_n-\omega)+\beta(\omega_n-\omega)^2=\int_{\mathbb{R}\times\mathbb{T}^N}(\omega_n-\omega)\widetilde{\nabla}_{e}\cdot(A\widetilde{\nabla}_{e}w)+\widetilde{\nabla}_{e_n}(\omega_n-\omega)A\widetilde{\nabla}_{e_n}\omega,
\end{equation*}
which converges to $0$ since $\omega\in H^1(\mathbb{R}\times\mathbb{T}^N)$, $\widetilde{\nabla}_{e}\cdot(A\widetilde{\nabla}_{e}w)\in L^2(\mathbb{R}\times\mathbb{T}^N)$ and  $\widetilde{\nabla}_{e_n}\omega_n\rightharpoonup\widetilde{\nabla}_{e_n}\omega$ weakly in $L^2(\mathbb{R}\times\mathbb{T}^N)$. By recalling our assumption (A4), one deduces that $||\omega_n-\omega||_{L^2(\mathbb{R}\times\mathbb{T}^N)}\to 0$ and $||\widetilde{\nabla}_{e_n}(\omega_n-\omega)||_{L^2(\mathbb{R}\times\mathbb{T}^N)}\to 0$ as $n\to +\infty$. Moreover, it follows from similar arguments in Step~1 that $||\partial_{\xi}(\omega_n-\omega)||_{L^2(\mathbb{R}\times\mathbb{T}^N)}\to 0$ and $||\nabla_y(\omega_n-\omega)||_{L^2(\mathbb{R}\times\mathbb{T}^N)}\to 0$ as $n\to +\infty$. Thus,
\begin{equation}\label{appendix1-3}
	||M_{c,e_n}^{-1}(g)-M_{c,e}^{-1}(g)||_{H^1(\mathbb{R}\times\mathbb{T}^N)}\to 0\text{ as }n\to+\infty.
\end{equation}
Eventually, the conclusion $M_{c_n,e_n}^{-1}(g_n)\to M_{c,e}^{-1}(g)$ in $H^1(\mathbb{R}\times\mathbb{T}^N)$ as $n\to +\infty$ holds according to (\ref{appendix1-1}), (\ref{appendix1-2}) and (\ref{appendix1-3}).

We now define some auxiliary operators similar to those in \cite{MR3772873}. For any $e\in\mathbb{S}^{N-1}, v\in H^2(\mathbb{R}\times\mathbb{T}^N), \vartheta\in\mathbb{R}$ and $\eta\in\mathbb{R}^N$, define
\begin{equation*}
	\begin{split}
	K_e(v,\vartheta,\eta):=&\eta A (\eta+2e)\partial_{\xi\xi}(U_e+v)+\eta A\nabla_y\partial_{\xi}(U_e+v)+\nabla_y\cdot(A\eta\partial_{\xi}(U_e+v))+\vartheta\partial_{\xi}(U_e+v)\\
	&+f(y,U_e+v)-f(y,U_e)+\beta v,
	\end{split}
\end{equation*}
and
\begin{equation*}
	G_e(v,\vartheta,\eta):=\left(v+M_{c,e}^{-1}(K_e(v,\vartheta,\eta)),\int_{\mathbb{R}^+\times\mathbb{T}^N}((U_e+v)^2-U_e^2)dyd\xi\right),
\end{equation*}
where $U_e, v$ denote $U_e(\xi,y), v(\xi,y)$ for short. By Step~1, we know that the function $G_e:H^2(\mathbb{R}\times\mathbb{T}^N)\times\mathbb{R}\times\mathbb{R}^N\to D\times\mathbb{R}$ and $G_e(0,0,0)=(0,0)$.

\textit{Step3: for every $c>0$ and $e\in\mathbb{S}^{N-1}$, the function $G_e$ is continuous and it is coutinuously Fr\'{e}chet differentiable with respect to $(v,\vartheta)$ and doubly continuously Fr\'{e}chet differentiable with respect to $\eta$.}

Since $K_e$ is affine with respect to $\vartheta$ and quadratic with respect to $\eta$, and since $f(y,u)$ is globally Lipschitz continuous in $u$, one gets the continuity of $K_e$. Then by the invertibility of $M_{c,e}$ and the Cauchy-Schwarz inequality, it follows that $G_e$ is continuous in $H^2(\mathbb{R}\times\mathbb{T}^N)\times\mathbb{R}\times\mathbb{R}^N$.

Since $G_e$ is quadratic with respect to $\eta$, it is elementary to get that $G_e$ is doubly continuously Fr\'{e}chet differentiable with respect to $\eta$. By the definition of the Fr\'{e}chet derivative, one can compute that for any $(v,\vartheta,\eta)\in H^2(\mathbb{R}\times\mathbb{T}^N)\times\mathbb{R}\times\mathbb{R}^N$ and $\tilde{\eta}\in\mathbb{R}^N$,
\begin{equation}\label{appendix1-4}
	\begin{split}
		G_e(v,\vartheta,\eta+\tilde{\eta})-G_e(v,\vartheta,\eta)=&\partial_{\eta}G_e(v,\vartheta,\eta)\tilde{\eta}+o(|\tilde{\eta}|)\\
		=&\Big(M_{c,e}^{-1}(2\partial_{\xi\xi}(U_e+v)(\eta+e)A\tilde{\eta}+\nabla_y\partial_{\xi}(U_e+v)A\tilde{\eta}\\
		&+\nabla_y\cdot(A\partial_{\xi}(U_e+v)\tilde{\eta})),0\Big)
	\end{split}
\end{equation}
Notice that $\lim\limits_{h\to 0}\left(f(y,U_e+u+hv)-f(y,U_e+u)\right)/h=f_u(y,U_e+u)v$ and $f_u(y,u)$ is globally Lipschitz continuous in $u$, one has that $G_e(v,\vartheta,\eta)$ is coutinuously Fr\'{e}chet differentiable with respect to $(v,\vartheta)$ and for any $(v,\vartheta,\eta)\in H^2(\mathbb{R}\times\mathbb{T}^N)\times\mathbb{R}\times\mathbb{R}^N$ and $(\tilde{v},\tilde{\vartheta})\in H^2(\mathbb{R}\times\mathbb{T}^N)\times\mathbb{R}^N$,
\begin{equation}\label{appendix1-5}
	\begin{split}
		G_e(v+\tilde{v},&\vartheta+\tilde{\vartheta},\eta)-G_e(v,\vartheta,\eta)=\partial_{(v,\vartheta)}G_e(v,\vartheta,\eta)(\tilde{v},\tilde{\vartheta})+o(||(\tilde{v},\tilde{\vartheta})||_{D\times\mathbb{R}})\\
		=&\Big(
\tilde{v}+M_{c,e}^{-1}(\eta A (\eta+2e)\partial_{\xi\xi}\tilde{v}+\eta A\nabla_y\partial_{\xi}\tilde{v}+\nabla_y\cdot(A\eta\partial_{\xi}\tilde{v})+\vartheta\partial_{\xi}\tilde{v}+\tilde{\vartheta}\partial_{\xi}(U_e+v)\\
&+f_u(y,U_e+v)\tilde{v}+\beta\tilde{v}),
2\int_{\mathbb{R}^+\times\mathbb{T}^N}(U_e+v)\tilde{v}dyd\xi
		\Big),
	\end{split}
\end{equation}
where $||(\tilde{v},\tilde{\vartheta})||_{D\times\mathbb{R}}=||\tilde{v}||_{D}+|\tilde{\vartheta}|$.

\textit{Step4: the Fr\'{e}chet differentiability of $(U_e,c_e)$.} 

For any $e\in\mathbb{S}^{N-1}$ and $(\tilde{v},\tilde{\vartheta})\in D\times\mathbb{R}$, define
\begin{equation*}
	Q_e(\tilde{v},\tilde{\vartheta}):=\left(\tilde{v}+M_{c,e}^{-1}(\tilde{\vartheta}\partial_{\xi}U_e+f_u(y,U_e)\tilde{v}+\beta\tilde{v}), 2\int_{\mathbb{R}^+\times\mathbb{T}^N}U_e\tilde{v}dyd\xi\right),
\end{equation*}
then $Q_e(\tilde{v},\tilde{\vartheta})=\partial_{(v,\vartheta)}G_e(v,\vartheta,\eta)(\tilde{v},\tilde{\vartheta})$ from (\ref{appendix1-5}). Here we emphasize that althought \eqref{1.1} is different from the equation studied in \cite{MR3772873}, the form of the operator $Q_e$ defined here is completely identical to that one defined in \cite[Lemma 2.10]{MR3772873}. Thus, it follows from \cite[Lemma 2.11]{MR3772873} that $Q_e:D\times\mathbb{R}\to D\times\mathbb{R}$ is invertible and there is $C>0$ such that
\begin{equation}\label{appendix1-6}
||Q_e^{-1}(g,d)||_{L^2(\mathbb{R}\times\mathbb{T}^N)}\leq C||(g,d)||_{D\times\mathbb{R}}
\end{equation}
for all $e\in\mathbb{S}^{N-1}, g\in L^2(\mathbb{R}\times\mathbb{T}^N)$ and $d\in\mathbb{R}$. Now we are ready to finish our proof by following the proof of \cite[Theorem 1.5]{MR3772873} step by step.

For any $b\in\mathbb{R}^N\setminus\{0\}$, let $U_b=U_{\frac{b}{|b|}}$ and $c_b=c_{\frac{b}{|b|}}$. Then $(U_b,c_b)$ is continuous with respect to $b$ and it satisfies
\begin{equation}\label{appendix1-7}
	c_b\partial_{\xi}U_b+\frac{b}{|b|}A\frac{b}{|b|}\partial_{\xi\xi}U_b+ \frac{b}{|b|}A\nabla_{x,y}\partial_{\xi}U_b+\nabla_{x,y}\cdot(A\frac{b}{|b|}\partial_{\xi}U_b)+\nabla_{x,y}\cdot(A\nabla_{x,y}U_b)+f(x,y,U_b)=0.
\end{equation}
Fix arbitrary $e\in\mathbb{S}^{N-1}$. For any $h\in\mathbb{R}^N$ such that $e+h\in\mathbb{R}^N\setminus\{0\}$, define $\tilde{U}_h:=U_{e+h}-U_e\in D$, $\tilde{c}_h:=c_{e+h}-c_e\in\mathbb{R}$ and $\tilde{h}:=(e+h)/|e+h|-e.$ By (\ref{property-1}), (\ref{appendix1-7}) and the definition of $G_e$, one can compute that $G_e(\tilde{U}_h,\tilde{c}_h,\tilde{h})=(0,0)$. Then it follows from $G_e(0,0,0)=(0,0)$ that
\begin{equation*}
	\begin{split}
		(0,0)&=G_e(\tilde{U}_h,\tilde{c}_h,\tilde{h})-G_e(0,0,0)\\
		&=\partial_{(v,\vartheta)}G_e(0,0,0)(\tilde{U}_h,\tilde{c}_h)+\partial_{\eta}G_e(0,0,0)\tilde{h}+\omega_1(\tilde{h})+\omega_2(\tilde{U}_h,\tilde{c}_h),
	\end{split}
\end{equation*}
where $\omega_1(\tilde{h})=o(|h|)$ and $\omega_2(\tilde{U}_h,\tilde{c}_h)=o(||(\tilde{U}_h,\tilde{c}_h)||_{L^2(\mathbb{R}\times\mathbb{T}^N)\times\mathbb{R}})$ as $|h|\to 0$. From (\ref{appendix1-1}) and  (\ref{appendix1-6}), one can prove that $Q_e^{-1}(\omega_1(\tilde{h}))=o(|h|)$ and $Q_e^{-1}(\omega_2(\tilde{U}_h,\tilde{c}_h))=o(|h|)$ as $|h|\to 0$. Therefore, by (\ref{appendix1-4}) and noticing that $\tilde{h}=-(e\cdot h)e+h+o(|h|)$ as $|h|\to 0$, one gets that
\begin{equation*}
	\begin{split}
		(U_{e+h}-U_e,c_{e+h}-c_e)&=(U_e',c_e')\cdot h+o(|h|)\\
		&=(e\cdot h)Q_e^{-1}(M_{c,e}^{-1}(2\partial_{\xi\xi}U_e eAe+\nabla_y\partial_{\xi}U_eAe+\nabla_y\cdot(A\partial_{\xi}U_ee)),0)\\
		&\quad -Q_e^{-1}(M_{c,e}^{-1}(2\partial_{\xi\xi}U_e eAh+\nabla_y\partial_{\xi}U_eAh+\nabla_y\cdot(A\partial_{\xi}U_eh)),0)+o(|h|),
	\end{split}
\end{equation*}
where $(U_e',c_e'): \mathbb{R}^N\to L^2(\mathbb{R}\times\mathbb{T}^N)\times\mathbb{R}$ denotes the Fr\'{e}chet derivative. By (\ref{appendix1-1}), (\ref{appendix1-6}) and the continuity of $(U_e,c_e)$ with respect to $e\in\mathbb{S}^{N-1}$ again, one can prove that $(U_{e_n}',c_{e_n}')\cdot h\to (U_e',c_e')\cdot h$ as $n\to +\infty$ when $e_n\to e$ as $n\to +\infty$. In other words, $(U_e,c_e)$ is first-order continuously Fr\'{e}chet differentiable with respect to $e\in\mathbb{S}^{N-1}$. Similar arguments can be applied to prove the second-order differentiability. We omit the details by referring to \cite{MR3772873}.

Taking derivatives with respect to $\xi$ and $y_i$ ($i=1,\dots,N$) on both sides of (\ref{property-1}) and considering those new equations of $\partial_{\xi}U_e, \partial_{ y_i}U_e$, one has that $\partial_{\xi}U_e$, $\partial_{ y_i}U_e$ ($i=1,\dots,N$) are also continuous with respect to $e\in\mathbb{S}^{N-1}$. On the other hand, one knows from Definition~\ref{pulsating-def} that $\lim\limits_{\xi\to\pm\infty}U_e(\xi,y)=0$, $1$, uniformly for $e\in\mathbb{S}^{N-1}$ and $y\in\mathbb{R}^N$, which further means that $\lim\limits_{\xi\to\pm\infty}U_e'\cdot h=0$ for any $h\in\mathbb{R}^N$, uniformly for $e\in\mathbb{S}^{N-1}$, $y\in\mathbb{R}^N$. Thus $||U_e'||:=\sup\limits_{|h|=1}||U_e\cdot h||_{L^2(\mathbb{R}\times\mathbb{T}^N)}$ is bounded uniformly for $e\in\mathbb{S}^{N-1}$. Similarly, one can get that $||\partial_{\xi}U_e'||, ||\nabla_yU_e'||, ||U_e''||, ||c_e'||, ||c_e''||$ are all bounded uniformly for $e\in\mathbb{S}^{N-1}$.

Finally, it follows from the same arguments in the proof of \cite[Lemma 2.7]{MR4334733} combined with the computation for the elliptic operator $\nabla\cdot(A(x)\nabla u)$  that (\ref{pulsating-lemma-4}) holds.
This completes the proof.
\end{proof}
\vskip 0.3cm

In order to prove Lemma \ref{pulsating-lemma-5}, we first show that the interfaces $(\Gamma_t)_{t\in\mathbb{R}}$ of sub-invasion of \eqref{1.1} cannot move infinitely fast. 

\begin{lemma}
	If $u$ is a sub-invasion of $0$ by $1$  with sets $(\Omega_t^{\pm})_{t\in\mathbb{R}}$ and $(\Gamma_t)_{t\in\mathbb{R}}$, then
	\begin{equation}\label{transition5}
		\sup\{d(x,\Gamma_{t-\tau}); t\in\mathbb{R}, x\in\Gamma_t\}<+\infty,
	\end{equation}
	for any $\tau>0$.
\end{lemma}
\begin{proof}
Take $\epsilon\in (0,\sigma)$ where $\sigma$ is given in (A3), let $M$ be given in (\ref{transition4}) and let $T_{\varepsilon}>0, R_{\varepsilon}>0$ be constants such that \cite[Lemma 3.1, Lemma 3.2]{MR3772873} hold. Since $\Gamma_t\subset \Omega_{t-\tau}^{-}$ for any $\tau>0$, we only have to show that the conclusion holds for $\tau\ge T_{\varepsilon}$. Take $R\ge \max\{R_{\varepsilon},(\overline{c}+\varepsilon)\tau\}$ where $\overline{c}=\sup_{e\in\mathbb{S}^{N-1}}c_e$.

Now assume by contradiction that
\begin{equation*}
	\sup\{d(x,\Gamma_{t-\tau}); t\in\mathbb{R}, x\in\Gamma_t\}=+\infty.
\end{equation*}
Then, there exist $t_0\in\mathbb{R}$ and $x_0\in\Gamma_{t_0}\subset \Omega_{t_0-\tau_0}^-$ such that 
\begin{equation}\label{5.1-1}
d(x_0,\Gamma_{t_0-\tau_0})\geq r_{M}+M+R,
\end{equation}
where $r_{M}>0$ such that
\begin{equation*}
	\sup\{d(y,\Gamma_{t_0}); y\in\Omega_{t_0}^{+}\cap B(x_0,r_{M})\}\geq M.
\end{equation*}
This is achievable since (\ref{transition2}). Then there is $y_0\in\mathbb{R}^N$ such that
\begin{equation}\label{5.1-2}
	y_0\in\Omega_{t_0}^+,\, |y_0-x_0|\le r_{M}\text{ and } d(y,\Gamma_{t_0})\geq M.
\end{equation}
By (\ref{transition4}), we know that
\begin{equation}\label{5.1-3}
	u(t_0,y_0)\geq 1-\varepsilon>1-\sigma>\sigma.
\end{equation}
On the other hand, one also has $B(y_0,R)\subset\Omega_{t_0-\tau}^-$ and $d(B(y_0,R),\Gamma_{t_0-\tau})\geq M$, which further implies
\begin{equation}\label{5.1-4}
	u(t_0-\tau,y)\leq\varepsilon\text{ for all }y\in B(y_0,R).
\end{equation}
For such $R$ and $\varepsilon$ given above, let $v_{R}(t,x)$ denote the solution of the Cauchy problem
\begin{equation*}
	\left\{
	\begin{aligned}
		&(v_R)_t-\nabla\cdot(A(x)\nabla v_R)=f(x,v_R), \quad t>0,x\in\mathbb{R}^N,\\
		&v_R(0,x)=\varepsilon\text{ for }|x|<R,\quad  v_R(0,x)=1\text{ for }|x|\geq R.
	\end{aligned}
	\right.
\end{equation*}
By \cite[Lemma 3.2]{MR3772873}, we know that 
$$v_R(t,x)\le \sigma, \hbox{ for all $T_{\varepsilon}\le t\le \frac{R}{\overline{c}+\varepsilon}$ and $|x|\le R-(\overline{c}+\varepsilon)t$}.$$
Since $u(t,x)$ is a sub-invasion and $R\ge (\overline{c}+\varepsilon)\tau$, it follows from (\ref{5.1-4}) and the comparison principle that 
\begin{equation*}
	u(t_0,y)\leq v_R(\tau,y-y_0)\text{ for all }y\in\mathbb{R}^N.
\end{equation*}
 Thus, $u(t_0,y_0)\leq v_R(\tau,0)\leq\sigma$ which contradicts (\ref{5.1-3}).
\end{proof}

\vskip 0.3cm

\begin{proof}[Proof of Lemma \ref{pulsating-lemma-5}]
This proof is inspired by \cite{MR2898886}, where the authors used the sliding method to show many properties of transition fronts. We only sketch the main steps as follows.

\textit{Step~1: Dividing $\mathit{\mathbb{R}^N}$ into several parts.} Define $\widetilde{u}_s(t,x)=\widetilde{u}(t+s,x), (t,x)\in\mathbb{R}\times\mathbb{R}^N$. It follows from (\ref{transition4}) that, for $0<\sigma<1/2$ given in (A3), there exist $M$ and $\widetilde{M}$ such that
\begin{equation}\label{appendix2-1}
	\left\{
	\begin{aligned}
		&\forall t\in\mathbb{R}, \forall x\in\Omega_{t}^{+}, (d(x,\Gamma_t)\geq M)\Rightarrow(u(t,x)\geq 1-\sigma/2),\\
		&\forall t\in\mathbb{R}, \forall x\in\Omega_{t}^{-}, (d(x,\Gamma_t)\geq M)\Rightarrow(u(t,x)\leq \sigma),\\		
	\end{aligned}
	\right.
\end{equation}
and
\begin{equation}\label{appendix2-2}
	\left\{
	\begin{aligned}
		&\forall t\in\mathbb{R}, \forall x\in\widetilde{\Omega}_{t}^{+}, (d(x,\widetilde{\Gamma}_t)\geq \widetilde{M})\Rightarrow(\widetilde{u}(t,x)\geq 1-\sigma/2),\\
		&\forall t\in\mathbb{R}, \forall x\in\widetilde{\Omega}_{t}^{-}, (d(x,\widetilde{\Gamma}_t)\geq \widetilde{M})\Rightarrow(\widetilde{u}(t,x)\leq \sigma).\\		
	\end{aligned}
	\right.
\end{equation}
Since $u$ and $\widetilde{u}$ are sub-invasion and sup-invasion of $0$ by $1$, and $\widetilde{\Omega}_{t}^{-}\subset \Omega_{t}^{-}$, there is $s_0>0$ large enough such that
\begin{equation*}
	\widetilde{\Omega}_{t+s}^{+}\supset\Omega_{t}^+\text{ and }d(\widetilde{\Gamma}_{t+s},\Gamma_t)\geq M+\widetilde{M}\text{ for all }t\in\mathbb{R}\text{ and }s\geq s_0.
\end{equation*}
Then one can deduce from above that
\begin{equation}\label{appendix2-3}
	(x\in\Omega_{t}^+)\text{ or }(x\in\Omega_{t}^-\text{ and }d(x,\Gamma_t)\leq M)\Rightarrow(\widetilde{u}_s(t,x)\geq 1-\sigma),
\end{equation}
Let 
$$\omega_M^-:=\{(t,x)\in\mathbb{R}\times\mathbb{R}^N; x\in\Omega_{t}^-\text{ and }d(x,\Gamma_t)\geq M\};\quad \omega_M^+:=\mathbb{R}\times\mathbb{R}^N\setminus\omega_M^-.$$
Then it follows from (\ref{appendix2-1}), (\ref{appendix2-2}) and (\ref{appendix2-3}) that
\begin{equation}\label{appendix2-4}
	\widetilde{u}_s(t,x)-u(t,x)\geq 1-2\sigma> 0\text{ on }\partial\omega_M^-.
\end{equation} 

\textit{Step~2: $\widetilde{u}_s\geq u\text{ in }\omega_M^-$ for all $s\geq s_0$.} Fix $s\geq s_0$ and define $\varepsilon^*=\inf\{\varepsilon>0; \widetilde{u}_s\geq u-\varepsilon\text{ in }\omega_M^-\}.$ Then $\varepsilon^*$ is a well-defined nonnegative number and we only need to prove that $\varepsilon^*=0$. Assume by contradiction that $\varepsilon^*>0$. There exists a sequence of real numbers $\{\varepsilon_n\}_{n\in\mathbb{N}}>0$ and a sequence of points $\{(t_n,x_n)\}_{n\in\mathbb{N}}$ in $\omega_M^-$ such that
\begin{equation}\label{appendix2-5}
	\varepsilon_n\to\varepsilon^*\text{ as }n\to+\infty\text{ and }\widetilde{u}_s(t_n,x_n)<u(t_n,x_n)-\varepsilon_n\text{ for all }n\in\mathbb{N}.
\end{equation}
The globally boundedness of $\nabla u$ combined with (\ref{appendix2-4}) and (\ref{appendix2-5}) imply that there is $\rho>0$ such that 
$$\liminf\limits_{n\to+\infty}d(x_n,\Gamma_{t_n})\geq M+2\rho.$$ 
Since $u$ is a sub-invasion of $0$ by $1$, one immediately knows that there is $n_0$ such that
\begin{equation}\label{appendix2-6}
	x\in\Omega_{t}^-\text{ and }d(x,\Gamma_t)\geq M\text{ for all }n\geq n_0,x\in \overline{B(x_n,\rho)}\text{ and }t\leq t_n.
\end{equation}

Thus, even if it means decreasing $\rho$, one can assume without loss of generality that $\rho<\tau$, where $\tau$ is given in {\color{red} (\ref{transition5}) }. Then there is a sequence $\{y_n\}_{n\in\mathbb{N}}$ in $\mathbb{R}^N$ such that
\begin{equation}\label{appendix2-7}
	y_n\in\Omega_{t_n-\tau+\rho}^-\text{ and }d(y_n,\Gamma_{t_n-\tau+\rho})=M+\rho\text{ for all }n\geq n_0.
\end{equation}
Define a $C^1$ path $P_n: [0,1]\to\Omega_{t_n-\tau+\rho}^-$ with $P_n(0)=x_n, P_n(1)=y_n$, one can combine (\ref{appendix2-6}), (\ref{appendix2-7}) and $1$ invades $0$ to get that, for each $n\geq n_0$, the set
\begin{equation*}
	E_n=[t_n-\tau,t_n]\times\overline{B(x_n,\rho)}\cup[t_n-\tau,t_n-\tau+\rho]\times\{x\in\mathbb{R}^N; d(x,P_n([0,1]))\leq\rho\}
\end{equation*}
is included in $\omega_M^-.$ As a consequence, 
\begin{equation*}
	v:=\widetilde{u}_s-(u-\varepsilon^*)\geq 0\text{ in }E_n\text{ for all }n\geq n_0.
\end{equation*}
Notice that $f(x,\cdot)$ is nonincreasing in $(-\infty,\sigma]$, one has that
\begin{equation*}
	v_t(t,x)-\nabla_x\cdot(A(x)\nabla v(t,x))+Lv(t,x)\geq 0\text{ in }E_n\text{ for all }n\geq n_0,
\end{equation*}
where $L=||f(x,u)||_{L^{\infty}(\mathbb{T}^N\times[0,1])}<+\infty$.

Let us focus on the distance between those sequences we defined just now. We claim that $\{d(x_n,\Gamma_{t_n})\}_{n\in\mathbb{N}}$ is bounded. Otherwise, up to extraction of subsequences, one has that $d(x_n,\Gamma_{t_n})\to+\infty$ and then $u(t_n,x_n), \widetilde{u}_s(t_n,x_n)\to 0$. However, it follows from (\ref{appendix2-5}) that $u(t_n,x_n)>\widetilde{u}_s(t_n,x_n)+\varepsilon_n\to\varepsilon^*>0$ as $n\to+\infty$, which is a contradiction. Moreover, one knows from (\ref{transition5}) that there is a sequence $\{\widetilde{x}_n\}_{n\in\mathbb{N}}$ in $\mathbb{R}^N$ such that $\widetilde{x}_n\in\Gamma_{t-\tau}$ for all $n\in\mathbb{N}$ and $\sup\{d(x_n,\widetilde{x}_n); n\in\mathbb{N}\}<+\infty$. Thus, for all $n\geq n_0$,
\begin{equation*}
	d(x_n,y_n)=d(x_n,\Gamma_{t_n-\tau+\rho})-(M+\rho)\leq d(x_n,\Gamma_{t_n-\tau})-(M+\rho)\leq d(x_n,\widetilde{x}_n)-(M+\rho)<+\infty.
\end{equation*}
Therefore, since $v(t_n,x_n)\to 0$ as $n\to+\infty$, it follows from the linear parabolic estimates that
\begin{equation*}
	v(t_n-\tau,y_n)\to 0\text{ as }n\to+\infty.
\end{equation*}
But, according to (\ref{appendix2-7}), there exists a sequence $\{z_n\}_{n\in\mathbb{N}}$ such that 
\begin{equation*}
	z_n\in\Omega_{t_n-\tau+\rho},\, d(z_n,y_n)=\rho\text{ and }d(z_n,\Gamma_{t_n-\tau+\rho})=M\text{ for all }n\geq n_0.
\end{equation*}
Since $u$ and $\nabla u$ are all globally bounded in $\mathbb{R}\times\mathbb{R}^N$ from standard parabolic interior estimates, even if it means decreasing $\rho$, one has that
\begin{equation}\label{appendix2-8}
	\rho\cdot(||u||_{L^{\infty}(\mathbb{R}\times\mathbb{R}^N)}+||\nabla_x u||_{L^{\infty}(\mathbb{R}\times\mathbb{R}^N)})\leq \varepsilon^*/2.
\end{equation}
Eventually, for all $n\geq n_0$, it follows from (\ref{appendix2-1}), (\ref{appendix2-3}) and (\ref{appendix2-8}) that
\begin{equation*}
	v(t_n-\tau,y_n)=\widetilde{u}_s(t_n-\tau,y_n)-u(t_n-\tau,y_n)+\varepsilon^*\geq 1-2\sigma+\varepsilon^*/2>0.
\end{equation*}
This reaches a contradiction, and hence $\varepsilon^*=0$.

\textit{Step~3: $\widetilde{u}_s\geq u\text{ in }\omega_M^+$ for all $s\geq s_0$.} This proof uses similar arguments as in Step~2, but the regions we are going to construct in $\omega_M^+$ are different from $E_n$ in Step~2, because $\omega_M^+$ is now  nondecreasing with respect to time $t$. Fix $s\geq s_0$ and define $\varepsilon_*=\inf\{\varepsilon>0; \widetilde{u}_s\geq u-\varepsilon\text{ in }\omega_M^+\}$. We now prove that $\varepsilon_*=0$. Assume by contradiction that $\varepsilon_*>0$. There exists a sequence of real numbers $\{\varepsilon_n\}_{n\in\mathbb{N}}>0$ and a sequence of points $\{(t_n,x_n)\}_{n\in\mathbb{N}}$ in $\omega_M^-$ such that
\begin{equation}\label{appendix2-9}
	\varepsilon_n\to\varepsilon_*\text{ as }n\to+\infty\text{ and }\widetilde{u}_s(t_n,x_n)<u(t_n,x_n)-\varepsilon_n\text{ for all }n\in\mathbb{N}.
\end{equation}
By similar arguments in Step~2, one gets that $\{d(x_n,\Gamma_{t_n})\}_{n\in\mathbb{N}}$ is bounded. From (\ref{transition5}), there is a sequence $\{\widetilde{x}_n\}_{n\in\mathbb{N}}$ in $\mathbb{R}^N$ such that $\widetilde{x}_n\in\Gamma_{t_n-\tau}$ for all $n\in\mathbb{N}$ and $\sup\{d(x_n,\widetilde{x}_n); n\in\mathbb{N}\}<+\infty$. By (\ref{transition2}), there is $r>0$ and $\{y_n\}_{n\in\mathbb{N}}$ in $\mathbb{R}^N$ such that
\begin{equation*}
	y_n\in\Omega_{t_n-\tau}^-,\, d(\widetilde{x}_n,y_n)=r\text{ and }d(y_n,\Gamma_{t_n-\tau})\geq M\text{ for all }n\in\mathbb{N}.
\end{equation*}
Then there exists $\{z_n\}_{n\in\mathbb{N}}$ in $\mathbb{R}^N$ such that $z_n\in\Omega_{t_n-\tau}^-$ and $d(z_n,\Gamma_{t_n-\tau})= M\text{ for all }n\in\mathbb{N}$. Combining the information above, it is easy to get that $\{d(x_n,z_n)\}_{n\in\mathbb{N}}$ is bounded.

Now choose $K\in\mathbb{N}\setminus\{0\}$ such that
\begin{equation*}
	K\rho\geq \max(\tau,\sup\{d(x_n,z_n;n\in\mathbb{N})\}).
\end{equation*}
For each $n\in\mathbb{N}$ and $0\leq i\leq K-1$, let
\begin{equation*}
	E_{n,i}=[t_n-\frac{i+1}{K}\tau,t_n-\frac{i}{K}\tau]\times \overline{B(X_{n,i},2\rho)},
\end{equation*}
where $\{X_{n,i}\}_{1\leq i\leq K}$ is a sequence in $\mathbb{R}^N$ such that
\begin{equation*}
	X_{n,0}=x_n, X_{n,K}=z_n\text{ and }d(X_{n,i},X_{n,i+1})\leq\rho\text{ for each }0\leq i\leq K-1.
\end{equation*}
One can check from (\ref{appendix2-9}) and the global boundedness of $\widetilde{u}_s-u$ and $\nabla(\widetilde{u}_s-u)$ that $w:=\widetilde{u}_s-(u-\varepsilon_*)<\varepsilon_*$ for large $n$. This implies that $E_{n,0}\subset\omega_M^+$ for large $n$, because we know from Step~2 that $w\geq \varepsilon_*$ in $\omega_M^-$. Since $f(x,\cdot)$ is nonincreasing in $[1-\sigma,+\infty)$, one can follow the similar arguments in Step~2 to use the linear parabolic estimates in $E_{n,0}$, and finally obtain that
\begin{equation*}
	w(t_n-\tau/K,X_{n,1})\to 0\text{ as }n\to+\infty.
\end{equation*}
An immediate induction yields $w(t_n-\tau,z_n)\to 0$ as $n\to+\infty$. However, we know from the definition of $z_n$ that $(t_n-\tau,z_n)\in\omega_M^-$, and hence $w(t_n-\tau,z_n)\leq\varepsilon_*$, which is a contradiction.

\textit{Step 4: Existence of the smallest $T$.} We know from Steps~1-3 that $\widetilde{u}_s\geq u$ in $\mathbb{R}\times\mathbb{R}^N \text{ for all }s\geq s_0.$ Now let
\begin{equation*}
	s_*:=\inf\{s\in\mathbb{R}; \widetilde{u}_s\geq u\text{ in }\mathbb{R}\times\mathbb{R}^N\}.
\end{equation*}
Then one has $s_*\leq s_0$ and $s_*>-\infty$ because $0<u(t,x)<1$ for all $(t,x)\in\mathbb{R}\times\mathbb{R}^N$ and $\widetilde{u}_s(t_0,x_0)\to 0$ as $s\to-\infty$ for all $(t_0,x_0)\in\mathbb{R}\times\mathbb{R}^N$. Thus, $s_*$ is well-defined and 
$$\widetilde{u}_{s_*}(t,x)=\widetilde{u}(t+s_*,x)\geq u(t,x)\text{ for all }(t,x)\in\mathbb{R}\times\mathbb{R}^N.$$
This gives that $s_*$ is the smallest $T$ we need, so we let $T=s_*$ in the following proof.

 Assume now that $\inf\{\widetilde{u}_T(t,x)-u(t,x); d(x,\Gamma_t)\leq M\}>0$. Then there is $\eta_0>0$ such that for any $\eta\in(0,\eta_0]$, $\widetilde{u}_{T-\eta}-u\geq 0$ for all $(t,x)\in\mathbb{R}\times\mathbb{R}^N$ such that $d(x,\Gamma_t)\leq M$. By the same arguments in Steps~2-3, one can prove that $\widetilde{u}_{T-\eta}-u\geq 0$ for all $(t,x)\in\mathbb{R}\times\mathbb{R}^N$, which contradicts the definition of $T$. Therefore, $\inf\{\widetilde{u}_T(t,x)-u(t,x); d(x,\Gamma_t)\leq M\}=0$ and consequently there exists a sequence $\{(t_n,x_n)\}_{n\in\mathbb{N}}$ of $\mathbb{R}\times\mathbb{R}^N$ such that
 \begin{equation*}
 	d(x_n,\Gamma_{t_n})\leq M\text{ for all }n\in\mathbb{N}\text{ and }\widetilde{u}_T(t_n,x_n)-u(t_n,x_n)\to 0\text{ as }n\to+\infty.
 \end{equation*}
This completes the proof.
\end{proof}

\section*{Acknowledgments}

This work is supported by the fundamental research funds for the central universities and NSF of China (No. 12101456, No. 12471201).

\end{document}